\documentclass{amsart}

\usepackage{amssymb, amsmath,mathtools}
\usepackage{mathrsfs,mathabx}
\usepackage{amscd}
\usepackage{verbatim}
\usepackage{stmaryrd}
\usepackage[dvipsnames]{xcolor}
\usepackage{subfig}

\usepackage{enumerate}
\usepackage{stmaryrd,esint}

\usepackage[colorlinks,linkcolor={blue},citecolor={blue},urlcolor={purple},]{hyperref}

\usepackage[colorinlistoftodos,prependcaption,textsize=tiny]{todonotes}

\theoremstyle{plain}
\newtheorem{theorem}{Theorem}[section]
\theoremstyle{remark}
\newtheorem{remark}[theorem]{Remark}

\theoremstyle{plain}

\newtheorem{lemma}[theorem]{Lemma}
\newtheorem{proposition}[theorem]{Proposition}

\newtheorem{definition}[theorem]{Definition}

\numberwithin{equation}{section}

\def\N{{\mathbb N}}
\def\Z{{\mathbb Z}}

\def\R{{\mathbb R}}

\newcommand{\E}{{\mathbf E}}
\renewcommand{\P}{{\mathbf P}}
\newcommand{\F}{{\mathscr F}}

\newcommand{\g}{\gamma}
\newcommand{\om}{\omega}
\renewcommand{\O}{\Omega}

\newcommand{\loc}{{\rm loc}}
\newcommand{\Tor}{\mathbb{T}}
\newcommand{\A}{\mathcal{A}}
\newcommand{\calL}{\mathscr{L}}

\newcommand{\one}{{{\bf 1}}}
\newcommand{\embed}{\hookrightarrow}
\newcommand{\s}{\delta}
\renewcommand{\emptyset}{\varnothing}
\newcommand{\wh}{\widehat}
\newcommand{\dd}{\mathrm{d}}
\newcommand{\Borel}{\mathscr{B}}

\newcommand{\vd}{v_{{\rm det}}}
\newcommand{\wwd}{w_{{\rm det}}}
\newcommand{\ud}{u_{{\rm det}}}

\newcommand{\T}{\Tor}

\newcommand{\p}{\mathbb{P}}
\newcommand{\q}{\mathbb{Q}}

\newcommand{\wt}{\widetilde}
\renewcommand{\i}{\mathrm{i}}
\newcommand{\Ls}{\mathbb{L}}
\newcommand{\Hs}{\mathbb{H}}
\newcommand{\Bs}{\mathbb{B}}

\newcommand{\Hsd}{\dot{\mathbb{H}}}

\newcommand{\D}{\mathcal{D}}
\newcommand{\supp}{\normalfont{\text{supp}}\,}
\newcommand{\Dom}{\mathcal{O}}
\newcommand{\LL}{\mathcal{P}_\theta}
\newcommand{\LLxi}{\mathcal{P}}
\newcommand{\LLn}{\mathcal{P}_{n}}
\newcommand{\LLnn}{\mathcal{P}_{\theta^n}}
\newcommand{\x}{\mathcal{V}}

\newcommand{\coeff}{\zeta}
\newcommand{\coeffone}{\coeff}
\newcommand{\Bi}{\mathcal{N}}

\allowdisplaybreaks

\begin{document}

\author{Antonio Agresti}
\address{Delft Institute of Applied Mathematics\\
Delft University of Technology \\ P.O. Box 5031\\ 2600 GA Delft\\The
Netherlands} 
\curraddr{Department of Mathematics Guido Castelnuovo, Sapienza University of Rome,
P.le Aldo Moro 5, 00185 Rome, Italy}
\email{antonio.agresti92@gmail.com}

\thanks{The author has received funding from the VICI subsidy VI.C.212.027 of the Netherlands Organisation for Scientific Research (NWO). The author is a member of GNAMPA (INdAM)}

\date\today

\title[Global smooth solutions by transport noise of 3D HSNE\lowercase{s}]{Global smooth solutions by transport noise \\ of 3D Navier-Stokes equations\\ with
small hyperviscosity}



\begin{abstract}
The existence of global smooth solutions to the Navier-Stokes equations (NSEs) with hyperviscosity $(-\Delta)^\g$ is open unless $\g $ is close to the J.-L.\ Lions exponent $ \frac{5}{4}$ at which the energy balance is strong enough to prevent singularity formation. If $1<\g\ll \frac{5}{4}$, then the global well-posedness of the hyperviscous NSEs is widely open as for the usual NSEs. 

In this paper, for all $\g>1$, we show the existence of a transport noise for which global smooth solutions to the stochastic hyperviscous NSEs on the three-dimensional torus exist with high probability. In particular, a suitable transport noise considerably improves the known well-posedness results in the deterministic setting. 
\end{abstract}

\maketitle
\setcounter{tocdepth}{1}

\tableofcontents

\section{Introduction}
In this paper, we investigate the effect of transport noise on the 3D hyperviscous Navier-Stokes equations (HNSEs in the following):
\begin{equation}
\label{eq:hyper_NS}
\left\{
\begin{aligned}
\partial_t u &+ (u\cdot\nabla)u 
+ (-\Delta)^{\g} u = -\nabla p \\
&  +\sqrt{\frac{3\mu}{2}} \sum_{k\in \Z^3_0}\sum_{\alpha\in \{1,2\}} \big[-\nabla \wt{p}_{k,\alpha}+ \theta_k(\sigma_{k,\alpha}\cdot\nabla) u \big]\circ \dot{W}^{k,\alpha}& \text{ on }&\T^3,\\
\nabla\cdot u&=0&\text{ on }&\T^3,\\
u(0,\cdot)&=u_0&\text{ on }&\T^3.
\end{aligned}
\right.
\end{equation}
Here, $u:[0,\infty)\times\O\times \T^3\to \R^3$ is the velocity field of an incompressible fluid, $p:[0,\infty)\times\O\times \T^3\to \R$ and $\wt{p}_{k,\alpha}:[0,\infty)\times\O\times \T^3\to \R$ denote the corresponding deterministic and turbulent pressure, respectively. 
The domain in which the fluid evolves is the three-dimensional periodic box 
$\T^3=[0,1]^3$ and $\g>1$ is the intensity of the hyperviscous operator $(-\Delta)^{\g}$ which is defined via Fourier transform as
\begin{equation}
\label{eq:symbol_F}
\mathcal{F}[(-\Delta)^{\g}u](k) = |k|^{2\g} \mathcal{F}[u](k)\ \  \text{ for }\ k\in \Z^3.
\end{equation} 
Moreover, $\theta=(\theta_k)_{k\in \Z^3_0}\in \ell^2$ where $\Z^3_0=\Z^3\setminus\{0\}$, $\mu>0$ is the noise intensity, $(\sigma_{k,\alpha})_{k\in \Z^3_0,\alpha\in\{1,2\}}$ is a suitable family of divergence-free vector fields, see Subsection \ref{ss:noise} below for details. Finally, $\circ$ denotes the Stratonovich product and $(W^{k,\alpha})_{k\in \Z^3_0,\alpha\in \{1,2\}}$ is a family of complex Brownian motions. Let us anticipate that the noise in \eqref{eq:hyper_NS} can be reformulated using only real-valued objects. However, we employ the complex formulation, as is the one used in the related literature. 

\smallskip

The (deterministic) HNSEs have attracted a lot of attention in the last decades as they present several similarities with the Navier-Stokes equations (NSEs in the following) corresponding to the case $\g=1$, although being a regularized version of the latter due to the hyperviscosity $\g>1$. As for the NSEs, in general, the only a priori estimate known up to the (possible) blow-up time for the HNSEs is the energy balance (see also \cite[Section 3]{CDLM20} for a local version). As shown by J.-L.\ Lions \cite{L69_lions}, the energy balance is sufficient to prevent the blow-up of strong solutions if $\g\geq \frac{5}{4}$. One way to understand the appearance of such a threshold is through scaling. Indeed, as recalled in the following (cf.\ the comments below \eqref{eq:scaling_HNSEs}), the energy space $L^2(\T^3;\R^3)$ is critical (resp.\ subcritical) for the HNSEs if and only if $\g= \frac{5}{4}$ (resp.\ $\g>\frac{5}{4}$). 
Several attempts have been made to extend the global well-posedness result of Lions below the threshold $\frac{5}{4}$. Firstly, Tao \cite{T09_supercritical} extended such global well-posedness to the case where $(-\Delta)^{5/4}$ is replaced by a logarithmically \emph{supercritical} operator.
The latter was later generalized by Barbato, Morandin, and Romito in \cite{BMR14} by replacing the logarithmic supercriticality with a Dini-type condition.
Later, Colombo, de Lellis, and Massaccesi in \cite{CDLM20}, by means of a generalized version of the Caffarelli-Kohn-Nirenberg theory of partial regularity of NSEs, showed the existence of $\varepsilon>0$ (depending only on the norm of $u_0$ in suitable function spaces) such that the HNSEs admit a global smooth solution for $\g\geq \frac{5}{4}-\varepsilon$. Another approach to such a result via the stability of solutions to the HNSEs was obtained by Colombo and Haffter in \cite{CH21}.
Finally, let us mention the work by Luo and Titi \cite{LT20_nonuniq}, in which indications about the sharpness of the J.-L.\ Lions exponent were found using convex integration methods (see also \cite{RT24_convex_integration} for related results).
 
\smallskip

Although many attempts have been made in the deterministic setting, the question of the existence of global smooth solutions of the HNSEs in the case $ 1<\g \ll \frac{5}{4}$ is still widely open. From a PDE point of view, one of the reasons is, as commented above, that the only a priori estimate for the HNSEs is the energy balance in $L^2(\T^3;\R^3)$ and the latter is (largely) supercritical. The same situation also appears for the case $\g=1$, which corresponds to the millennium problem on the existence of global smooth solutions for the NSEs \cite{F00_NSproblem}. 

The following is an informal version of Theorem \ref{t:main_theorem_L2}. It shows that transport noise can considerably improve the well-posedness theory of the HNSEs compared to the deterministic setting. In particular, the following can be seen as a \emph{regularization by noise} result for the HNSEs. 
The physical relevance of transport noise in the context of HNSEs is discussed at the beginning of Subsection \ref{ss:regularity_by_noise} below.

\begin{theorem}[Global smooth solutions of HNSEs by transport noise -- Informal statement]
\label{t:intro}
Let 
$$
\g>1
$$
and fix a (bounded) ball $\mathcal{B}$ in $H^{1}(\T^3;\R^3)$. Then, 
for all $\varepsilon\in (0,1)$, there exist $\mu>0$, $
\theta\in \ell^2(\Z^3_0) $ and a family of divergence-free vector fields $(\sigma_{k,\alpha})_{k\in \Z^3_0,\alpha\in \{1,2\}}$ such that for all 
divergence-free vector field $u_0\in \mathcal{B}$ there exists a unique \emph{smooth} solution $u$ to the {\normalfont{HNSEs}} with transport noise \eqref{eq:hyper_NS} that is \emph{global in time} with probability larger than $1-\varepsilon$.
\end{theorem}

The choice of the vector fields $\sigma_{k,\alpha}$ is canonical and is given in Subsection \ref{ss:noise} while, for the one of $\theta\in \ell^2$, the reader is referred to 
Remark \ref{r:generalization} for some comments.

\smallskip

One of the main difficulties behind the proof of Theorem \ref{t:intro} is the subtle way in which transport noise affects the dynamics of the HNSEs and improves the available up-to-the-(possible)-explosion-time bounds compared to the deterministic setting. 
For instance, transport noise does \emph{not} alter the energy balance. Indeed,
since $\sigma_{k,\alpha}$ is divergence-free, it follows that, a.s.\ and for all $t<\tau$, 
\begin{equation}
\label{eq:energy_estimate_unchanged}
\frac{1}{2}\,\|u(t)\|_{L^2(\T^3;\R^3)}^2 + \int_0^t \|(-\Delta)^{\g/2} u(s)\|_{L^2(\T^3;\R^3)}^2\, \dd s =
\frac{1}{2}\,\|u_0\|_{L^2(\T^3;\R^3)}^2 . 
\end{equation} 
The above is precisely the energy balance observed in the deterministic setting.
Here, $\tau$ represents the explosion time of $u$ in the critical space $H^{5/2-2\g}(\T^3; \R^3)$ for well-posedness of \eqref{eq:hyper_NS} (however, as shown in Theorem \ref{t:local_HNSEs_hyper_small_viscosity}, the explosion time in $H^\delta(\T^3; \R^3)$ is independent of $\delta \geq \frac{5}{2} - 2\g$). A discussion on criticality is provided below \eqref{eq:scaling_HNSEs}.
The key to proving Theorem \ref{t:intro} lies in capturing the \emph{enhanced dissipation} effect of transport noise. We are able to rigorously capture this enhancement only when $\theta$ is carefully chosen, though there is some flexibility in its selection (see Remark \ref{r:generalization}). Further comments are given in Subsection \ref{ss:novelty}.

Before going further, let us give some more comments on Theorem \ref{t:intro}.
In the latter, solutions to the stochastic HNSEs \eqref{eq:hyper_NS} are smooth in the following sense: 
\begin{align*}
u\in  C_{\loc}^{\vartheta_0,\vartheta_1}((0,\tau)\times \T^3;\R^3)\ \text{ a.s.\ for all }\ \vartheta_0<\tfrac{1}{2}, \ \vartheta_1<\infty,
\end{align*}
where $C^{\vartheta_0,\vartheta_1}$ denotes the set of all maps that are locally $\vartheta_0$- (resp.\ $\vartheta_1$-)H\"older continuity in time (resp.\ space) uniformly in space (resp.\ time). The reader is referred to Subsection \ref{ss:notation} for a precise definition. Further regularity properties of solutions to \eqref{eq:hyper_NS} are given in Theorem \ref{t:local_HNSEs_hyper_small_viscosity}. 
Of course, in case $\g\geq \frac{5}{4}$, Theorem \ref{t:intro} can be improved, see Theorem \ref{t:local_HNSEs_hyper_high_viscosity}.
Finally, 
the choice of the space $H^1( \T^3;\R^3)$ in Theorem \ref{t:intro} is only for exposition convenience. Indeed, as Theorem \ref{t:main_theorem_L2} shows that the space $H^1( \T^3;\R^3)$ in the above result can be replaced by any \emph{subcritical} spaces of the form 
\begin{equation}
\label{eq:class_initial_data_revision}
H^{\delta}(\T^3;\R^3)\quad \text{ with } \quad \delta>\tfrac{5}{2}-2\g,
\end{equation}
see Subsection \ref{ss:novelty} for details on criticality. As commented in Remark \ref{r:generalization}, our arguments are flexible enough to deal with even rougher initial data. However, for the sake of clarity, we prefer to present our arguments in the simplest setting at the expense of the level of generality.

\smallskip

In the author's opinion, one of the interesting features of Theorem \ref{t:intro} is that, as $\g>1$ is arbitrary, it comes close to a regularization by noise result for 3D NSEs with a physical noise. Compared with the existing literature, the techniques used in this manuscript are robust enough to tackle the limiting case of NSEs, i.e., $\g=1$. 
Further comments on this are given in Subsections \ref{ss:regularity_by_noise} and \ref{ss:towards_NS} below.

\smallskip

For exposition convenience, we now split the discussion into several subsections.
Firstly, in Subsection \ref{ss:regularity_by_noise}, we provide some comments on the physical relevance of transport noise in fluid dynamics and a brief overview of regularization by noise results via transport noise. 
Secondly, in Subsection \ref{ss:novelty}, we discuss the novelties and the strategy used in the proof of Theorem \ref{t:intro}. In Subsection \ref{ss:towards_NS} we discuss the possibilities and challenges of extending Theorem \ref{t:intro} to the case of 3D NSEs. Finally, in Subsections \ref{ss:overview} and 
\ref{ss:notation}, we give an overview of the current manuscript and fix the notation, respectively.

\subsection{Physical relevance of transport noise and its regularizing properties}
\label{ss:regularity_by_noise}
Transport noise is by now a well-established model in stochastic fluid dynamics, see e.g., \cite{BCF91,BCF92,CP01,FlGa,FL32_book,FP20_small,HLN19,HLN21_annals} and the references therein. Rigorous justifications are given in \cite{DP22_two_scale,FlaPa21,M14_derivation,MR01,MR04}. 
The reader is referred to \cite[Subsection 1.2]{FL19} for a heuristic derivation of fluid dynamics models via separation of scales (see also \cite[Subsection 1.2]{A24_anomalous} for comments on the physical meaning of $\mu$). 
In the latter scenario, transport noise can be seen as an idealization of the effect of the \emph{small scales} of a fluid in a turbulent regime. Such derivation extends almost verbatim to the HNSEs, and therefore, we omit the details.

The possibility of preventing singularity formation in PDEs (and especially in models from fluid dynamics) via physically relevant noises has attracted a lot of attention in the last few decades.
It is not possible to give a complete overview here, and we mainly focus on the most relevant works for our purposes. In any case, we would like to mention the seminal work \cite{FGP10} due to Flandoli, Gubinelli and Priola on well-posedness by transport noise for transport equations and the recent result by Coghi and Maurelli \cite{CM23} where they showed that transport noise restores the uniqueness of the 2D Euler equation on the whole space in case of $L^p$-initial data for suitable $p<\infty$.
Coming back to our discussion, our approach can be seen as a refinement of that of Flandoli and Luo \cite{FL19}, where they obtained the global well-posedness of the 3D NSEs in \emph{vorticity} formulation by transport noise. 
Details on the novelties of our approach are given in Subsection \ref{ss:novelty} below. 
Let us recall that transport noise is not the physically relevant stochastic perturbation of the 3D NSEs in vorticity formulation, cf.\ \cite[Subsection 1.2 and Appendix 3]{FL19}. Indeed, although transport noise is well-motivated at the level of the velocity, in the case of vorticity formulation, a stochastic stretching term appears, and the arguments used in \cite{FL19} fail in the latter situation, see Subsection 1.2 and Appendix 2 there. Thus, the work \cite{FL19} does not provide a regularization by noise result for the 3D NSEs with a physically relevant noise, see \cite[Appendix 3]{FL19} for more on this.
The arguments in \cite{FL19} extend to a nonlinear setting the one in \cite{G20_convergence} by Galeati, where the possibility of enhancing the dissipation of a linear PDE via transport noise was first shown. 
After these works, the effect of transport noise on mixing, enhanced dissipation, and anomalous dissipation has been widely investigated, see e.g., \cite{A24_anomalous,FGL22_eddy,FGL21_quantitative,FL23_Boussinesq,FLL24,HPZZ23,L23_enhanced,LTZ24} and the references therein. Related results on regularization by noise can be found in \cite{A22,FGL21,FHLN20,L23_regularization}. 

\smallskip

It is worth noticing that the variational approach to stochastic fluid models by Holm \cite{H15_SVP} leads to a noise that not only contains a transport term, but also additional terms depending on the velocity field $u$. This noise has the advantage of being circulation-preserving, but it does not preserve energy, i.e., the identity \eqref{eq:energy_estimate_unchanged} does not hold if we consider these additional terms in the noise. A discussion on various stochastic models can be found in  \cite[Section 1]{BPW25} and \cite[Section 2]{DM25_PhysD}. 
Regularization properties of this type of noise are established in \cite{BFL24_magnetic,BFLT24,BFL24_magnetic2}.

\smallskip

To conclude this subsection, let us also mention that the recent results on existence of (very) weak solutions to Navier-Stokes equations in the presence of transport noise via stochastic convex integration \cite{HLP24_convex,P23_convex_NS} have raised some doubts about the possibility of improving well-posedness by such a noise in the case of 3D NSEs, see \cite{HB22_tame} for a discussion.
For results on HNSEs, the reader is referred to \cite{LRS24_LR,RS23_convex_hypo} for the case of hypoviscosity (i.e., $\g<1$) and to \cite{LT20_nonuniq} for the deterministic setting in case of hyperviscosity.
The current manuscript contributes to this picture from the opposite side. While transport noise cannot improve well-posedness in a (very) weak setting, it can prevent the buildup of singularities of \emph{strong} solutions of Navier-Stokes equations with hyperviscosity $\g>1$. Further discussion on the Navier-Stokes case $\g=1$ is postponed to Subsection \ref{ss:towards_NS}. 

\subsection{Novelties and strategy -- Improving scaling limits via $L^p$-estimates}
\label{ss:novelty}
In this subsection, we describe the novelty and the key ingredients in the proof of Theorem \ref{t:intro}.
The main novelties of our approach compared to the existing literature concern the use of stochastic maximal $L^p$-regularity estimates \emph{in time} and (sub)criticality in certain scaling limit arguments, cf.\ \cite[Theorem 1.4]{FL19} and Theorem \ref{t:scaling_limit_HNSE} below. 
The latter result is the key step in proving Theorem \ref{t:intro} and encodes the enhanced dissipation effect of transport noise.
In the following, we assume some familiarity with the main argument behind the scaling limits as in \cite[Theorem 1.4]{FL19} or Theorem \ref{t:scaling_limit_HNSE}.
As the discussion below shows, the need for the $L^p$-theory in the scaling limit is due to the supercriticality of the energy space $L^2$.

Before going into the details, let us collect some more remarks. Firstly, stochastic maximal $L^p$-regularity concerns optimal space-time regularity estimates for the \emph{linearized} version of \eqref{eq:hyper_NS}. Thus, proving the latter is (much) easier than obtaining estimates for the original nonlinear problem \eqref{eq:hyper_NS}.
The reader is referred to Section \ref{app:smr} and \cite[Subsection 1.2]{AV19_QSEE1} for details and references on stochastic maximal $L^p$-regularity. Secondly, some use of $L^p$-estimates was already present in \cite{A22,A24_anomalous}. However, in the current manuscript, the key role of the $L^p$ in \emph{time} estimates has been fully understood in the context of scaling limits.

\smallskip

Now, let us come back to the description of the novelties. 
The scaling limit result of Theorem \ref{t:scaling_limit_HNSE} examines the behaviour of solutions to \eqref{eq:hyper_NS} as $n\to \infty$ with 
\begin{equation}
\label{eq:theta_behaviour}
\|\theta^n\|_{\ell^\infty}\to 0 \qquad \text{ and } \qquad \|\theta^n\|_{\ell^2}\equiv 1;
\end{equation}
cf.\ \eqref{eq:choice_theta_n}-\eqref{eq:property_sequence_thetan}.
More precisely, Theorem \ref{t:scaling_limit_HNSE} establishes that the solution $u^n$ to (a modified version of) \eqref{eq:hyper_NS} with $\theta=\theta^n$ convergences in suitable function spaces to $\ud$ that solves the deterministic HNSEs with \emph{increased} viscosity:
\begin{equation}
\label{eq:hyper_NS_intro_increased}
\left\{
\begin{aligned}
\partial_t \ud &+ (\ud\cdot \nabla) \ud + (-\Delta)^\g \ud=\frac{3\mu}{5}\Delta \ud -\nabla p & \text{ on }&\T^3,\\
\nabla\cdot \ud&=0&\text{ on }&\T^3,\\
\ud(0,\cdot)&=u_0&\text{ on }&\T^3.
\end{aligned}
\right.
\end{equation} 

The limit as $n\to \infty$ for which \eqref{eq:theta_behaviour} holds is usually referred to as \emph{scaling limit} since, in a certain sense, corresponds to `zoom out from the small scales' that are modelled by the transport noise (see the discussion at the beginning of Subsection \ref{ss:regularity_by_noise}). Thus, as explained \cite[Subsection 2.2]{A22}, the scaling limit of Theorem \ref{t:scaling_limit_HNSE} can be understood as a homogenization result for the SPDE \eqref{eq:hyper_NS}, and \eqref{eq:hyper_NS_intro_increased} can be thought of as the `effective equation' for the problem \eqref{eq:hyper_NS}.

Note that the energy balance for solutions to \eqref{eq:hyper_NS_intro_increased} differs from that of solutions to \eqref{eq:hyper_NS}, i.e., \eqref{eq:energy_estimate_unchanged}. In particular, the convergence $u^n\to \ud$ \emph{cannot} hold in spaces with high smoothness. Further comments are given below \eqref{eq:smr_p_intro}.

\smallskip

To fully exploit the scaling limit argument described above for the stochastic HNSEs \eqref{eq:hyper_NS} with $1<\g\ll \frac{5}{4}$, one needs the following ingredients:
\begin{itemize}
\item  $u^n$ converges to $\ud$ in sufficiently `rich' function spaces and in which local well-posedness of the HNSEs holds.
\item 
The existence of $\mu\gg1$ for which \eqref{eq:hyper_NS_intro_increased} is globally well-posed.
\end{itemize}
As mentioned at the beginning of this subsection, the above points will be addressed by using
stochastic maximal $L^p$-time regularity and (sub)criticality.
To explain the need for $L^p$-techniques and the role of subcriticality, let us begin by discussing the scaling of \eqref{eq:hyper_NS}. 
As it is well known, the scaling analysis allows us to understand the minimal smoothness assumption needed for local well-posedness. 
One can check that, on $\R^3$, solutions to the deterministic HNSEs are invariant under the rescaling:
\begin{equation}
\label{eq:scaling_HNSEs}
u(t,x)\mapsto u_{\lambda} (t,x)\stackrel{{\rm def}}{=} \lambda^{1-\frac{1}{2\g}}u(\lambda t,\lambda^{\frac{1}{2\g}} x), \ \text{ for } (t,x)\in\R_+\times\R^3,
\end{equation} 
and $\lambda>0$. The above induces a map on the set of the initial data given by 
\begin{equation}
\label{eq:u0_u0lambda}
u_0\mapsto u_{0,\lambda}\stackrel{{\rm def}}{=} \lambda^{1-\frac{1}{2\g}}u_0(\lambda^{\frac{1}{2\g}} \cdot).
\end{equation}
Within the scale $(\dot{H}^{\delta}(\R^3;\R^3))_{\delta\in\R}$ of homogeneous Sobolev spaces, the only space that is invariant under the map in \eqref{eq:u0_u0lambda} is 
$
\dot{H}^{5/2-2\g}(\R^3) 
$. Indeed, for $\s_\g=5/2-2\g$,
$$
\|u_{0,\lambda}\|_{\dot{H}^{\s_\g}(\R^3;\R^3) }
=
\lambda^{1-\frac{1}{2\g}}\lambda^{\frac{1}{2\g}(\s_\g-\frac{3}{2})}\|u_{0}\|_{\dot{H}^{\s_\g}(\R^3;\R^3) }
= \|u_{0}\|_{\dot{H}^{\s_\g}(\R^3;\R^3) }.
$$
In the PDE literature, if such an invariance holds, then the corresponding space is called \emph{critical} for the given PDE.
In case $\R^3$ is replaced by $\T^3$ (or even a domain), the above argument works `locally in space', and therefore we can still refer to $H^{5/2-2\g}(\T^3;\R^3)$ as the critical for \eqref{eq:hyper_NS} and \eqref{eq:hyper_NS_intro_increased}. Interestingly, the criticality of such a space also holds if one uses the abstract theory of critical spaces due to Pr\"uss, Simonett, and Wilke \cite{CriticalQuasilinear}. The reader is referred to Theorem \ref{t:local_HNSEs_hyper_small_viscosity} for an application of the stochastic version of the latter theory developed in \cite{AV19_QSEE1,AV19_QSEE2}.
 
Coming back to the scaling limit argument, the convergence of $u^n$ to $\ud$ holds in sufficiently `rich' function spaces if and only if they are at least as `smooth' as $H^{5/2-2\g}(\T^3)$ (this fact can be made rigorous by looking at Sobolev indexes). For instance, in our manuscript, we measure the convergence of $u^n$ to $\ud$ in the space
\begin{equation}
\label{eq:convergence_CH}
C([0,T];H^{\delta}(\T^3;\R^3)) \quad \text{ for } \ \tfrac{5}{2}-2\g<\delta<\g,
\end{equation}
where $T<\infty$ is arbitrary. We prove the convergence of $u^n$ to $\ud$ in \eqref{eq:convergence_CH} via a compactness argument. This explains the subcritical condition $\delta>\frac{5}{2}-2\g$ as, in the latter regime, we can 
lose some regularity to obtain compactness while still being in sufficiently `rich' function spaces for local well-posedness.

The \emph{enemy} for the convergence $u^n$ to $\ud$ is the loss of regularity of the noise coefficients as $n\to \infty$. Indeed, if \eqref{eq:theta_behaviour} holds, then one can check that (see Subsection \ref{ss:noise} for the definition of $\sigma_{k,\alpha}$)
\begin{equation}
\label{eq:lack_regularity_coeff}
\sup_{n\geq 1} \|(\theta_k^n\sigma_{k,\alpha})_{k,\alpha}\|_{L^{\infty}(\T^3;\ell^2)}<\infty 
\quad \text{ and }\quad
\sup_{n\geq 1} \|(\theta_k^n\sigma_{k,\alpha})_{k,\alpha}\|_{H^{r}(\T^3;\ell^2)}=\infty
\end{equation}
for all $r>0$, see \cite[Proposition 4.1]{Arole25}.
Hence, to prove an a priori estimate for $u^n$ that is uniform in $n\geq 1$, the key is to use estimates which are valid in the presence of \emph{rough} $L^\infty$-transport noise and are strong enough to obtain an a priori bound in the space \eqref{eq:convergence_CH}. This is exactly what happens by using stochastic maximal $L^p$-regularity in time, cf.\ Section \ref{app:smr}. 
To see the key role of the $p$-integrability in time with $p\gg 2$, let us note that, the `maximal $L^p$-regularity space' for the \emph{weak setting} in the spatial variable $x\in \T^3$ of the HSNEs \eqref{eq:hyper_NS} is given by (cf.\ \eqref{eq:main_theorem_L2_regularity_1} in Theorem \ref{t:main_theorem_L2})
\begin{equation}
\label{eq:smr_p_intro}
H^{\vartheta,p}(0,T;H^{\g(1-2\vartheta)}(\T^3;\R^3))\cap 
L^p(0,T;H^{\g}(\T^3;\R^3)) \ \text{ for all } \ \vartheta\in [0,1/2).
\end{equation}
Here, we use the weak setting in the space variable as it is the only one that does not require any smoothness from the noise coefficients, and hence it is compatible with \eqref{eq:lack_regularity_coeff}.
Note that, by the trace theory of anisotropic spaces \cite[Theorem 1.2]{ALV21}, the space \eqref{eq:smr_p_intro} 
embeds into
\begin{equation}
\label{eq:convergence_CB}
C([0,T];B^{\g(1-2/p)}_{2,p}(\T^3;\R^3))
\end{equation}
where $B$ denotes a Besov space. Their appearance is due to standard results for real interpolation, see e.g., \cite[Chapter 5]{BeLo}. Note that, for all $\frac{5}{4}-2\g<\delta<\g$, we can choose $p<\infty$ such that 
$$
B^{\g(1-2/p)}_{2,p}(\T^3;\R^3)\embed H^{\delta}(\T^3;\R^3).
$$ 
Therefore, the maximal $L^p$-regularity in time and weak (in the PDE sense) in space is suitable to prove an a priori estimate that is uniform in $n$, still being compatible with the lack of uniform smoothness of the noise coefficients \eqref{eq:lack_regularity_coeff}.
Let us warn the reader that it is impossible to avoid the weak setting in the spatial variable.  
Indeed, even in the case of a scalar linear SPDE, if $(\theta^n)_{n}$ satisfies \eqref{eq:theta_behaviour} and the initial data $u_0$ is not constant, then a uniform in $n$ estimate along the scaling limit in a space stronger than the weak one leads to a contradiction with the energy balance of the limiting deterministic problem with increased viscosity \eqref{eq:hyper_NS_intro_increased}. The reader is referred to either \cite[Remark 3.4]{A24_anomalous} or \cite[Subsection 4.4.2]{Arole25} for details. 
Thus, the $ L^p$-approach used here is unavoidable.

Next, we discuss the global well-posedness of \eqref{eq:hyper_NS_intro_increased} for $\mu \gg 1$. Although the additional (eddy) dissipation term $\mu\Delta$ is of lower order compared to $(-\Delta)^\g$, its influence for large $\mu$ guarantees the global well-posedness of \eqref{eq:hyper_NS_intro_increased}. This result can be established within a Hilbert space framework, and the detailed proof is provided in Section \ref{s:global_high_viscosity}.

\smallskip

 Finally, let us mention that our class of initial data is larger compared to the one used in \cite{FL19}.  Indeed, as discussed around \eqref{eq:class_initial_data_revision}, we can allow initial data with vorticity of \emph{negative} regularity, i.e., $\mathrm{curl}\, u_0\in H^{\chi}(\T^3;\R^3)$ with $\chi>\frac{3}{2}-2\g$. Note that the regularity threshold $\frac{3}{2}-2\g$ is critical for the vorticity formulation. However, we stress that taking the $\mathrm{curl}$ in the HNSEs \eqref{eq:hyper_NS} does not yield the stochastic Navier-Stokes equations considered in \cite{FL19} (even in the case $\g=1$). Therefore, it is not possible to recover the results in the previously mentioned reference from the present manuscript.

\subsection{An eye towards the 3D Navier-Stokes equations}
\label{ss:towards_NS}
Unfortunately, we are not able to cover the case of the 3D NSEs, i.e., $\g=1$. 
However, in contrast to the approach via vorticity formulation in \cite{FL19} (see Appendix 2 there), there are several indications that our methods could also extend to the 3D NSEs with transport noise.
To elaborate on this point, let us recall that, as described in Subsection \ref{ss:novelty}, the central tool is the maximal $L^p$-regularity in time estimates of Section \ref{app:smr}. The reader can check that if Theorem \ref{t:smr_theta} were also true in the case $\g=1$, then the methods of the current work imply that Theorem \ref{t:intro} also holds in the case $\g=1$, i.e., for the 3D NSEs.
However, the case $\g=1$ of Theorem \ref{t:smr_theta} is more complicated than the case $\g>1$, as in the latter case, the transport noise acts as a lower-order term w.r.t.\ the leading differential operators. Thus, if $\g=1$ and $p\gg 2$  (see \cite[Theorem 3.2]{A24_anomalous} for $p\sim 2$), the proof of Theorem \ref{t:smr_theta} requires much more carefulness, and in some sense, $\g=1$ can be seen as a critical value for the estimate of Theorem \ref{t:smr_theta} and thus more sophisticated arguments are needed.

Next, we discuss the potential validity of the claims in Theorem \ref{t:smr_theta} in the critical case $\g=1$, for $p\gg 2$ (note that the case $p=2$ is an easy consequence of It\^o's formula, see e.g., \cite[Lemma 4.1]{AV24_variational}).
The main source of hope for its validity is that the leading operators in the corresponding Stokes equations (SEs in the following) with transport noise (i.e., \eqref{eq:hyper_NS_linear} with $\g=1$) are time-independent. This situation seems comparable with the case of maximal $L^p$-regularity estimates in the case of autonomous (i.e., time-independent) deterministic problems.
In the deterministic setting, via semigroup methods, it is well known that maximal $L^p$-regularity estimates for autonomous problems are \emph{independent} of $p$; see e.g., \cite{Dore}, \cite[Chapter 4]{pruss2016moving} or \cite[Chapter 17]{Analysis3}.
In other words, maximal $L^p$-regularity extrapolates from \emph{one} $p\in (1,\infty)$ to \emph{all} $p\in (1,\infty)$.
The proof of such extrapolation results typically involves harmonic analytic tools such as Calder\'on-Zygmund or Littlewood-Paley theory. 
In the stochastic case, recent results in the application of harmonic analysis to SPDEs have been obtained in \cite{KK20,LV21_singular}. However, in both cases, a kernel representation of the solution is needed, and the latter is not available for the SEs with transport noise. Therefore, their results do not apply to our situation. 
Finally, let us mention that, in the case of non-autonomous problems, maximal $L^p$-regularity does \emph{not} extrapolate \cite{BMV24}. The lack of extrapolation in non-autonomous PDEs is due to a wild in-time behaviour of the differential operator coefficients, and the latter is expected not to be displayed by a time-dependent noise.
Finally, it is worth pointing out that the arguments in the current manuscript show that stochastic $L^p$-regularity for the SEs with transport noise is required only for a \emph{single} value $p>4$. This fact can potentially allow for other variations of $L^p$-extrapolation, see e.g., \cite[Subsection 4.2]{S18_book}.

In any case, we hope the current work might stimulate further research in the context of extrapolation of maximal $L^p$-regularity estimates for SPDEs.

\subsection{Overview}
\label{ss:overview}
This work is organized as follows:

\begin{itemize}
\item Section \ref{s:statement}: Basic notation and statement of the main result -- Theorem \ref{t:main_theorem_L2}.
\item Section \ref{app:smr}: Stochastic maximal $L^p$-regularity estimates for HSEs with rough transport noise.
\item Section \ref{s:local_well_stoch_problem}: Existence, regularity and global well-posedness for small data for the stochastic HNSEs \eqref{eq:hyper_NS}.
\item Section \ref{s:global_high_viscosity}: Global well-posedness of \eqref{eq:hyper_NS_intro_increased} with high-viscosity $\mu\gg1$.
\item Section \ref{s:scaling_limit}: Proof of Theorem \ref{t:main_theorem_L2} via scaling limit in an $L^p$-setting.
\end{itemize}

\subsection{Notation}
\label{ss:notation}
Here, we collect the basic notation used in the manuscript.  
For given parameters $p_1,\dots,p_n$, we write $R(p_1,\dots,p_n)$ if the quantity $R$ depends only on $p_1,\dots,p_n$.
For two quantities $x$ and $y$, we write $x\lesssim y$, if there exists a constant $C$ such that $x\le Cy$. If such a $C$ depends on above-mentioned parameters $p_1,\dots,p_n$ we either mention it explicitly or indicate this by writing $C{(p_1,\dots,p_n)}$ and correspondingly $x\lesssim_{p_1,\dots,p_n}y$ whenever $x\le C{(p_1,\dots,p_n)}y$. We write $x\eqsim_{p_1,\dots,p_n} y$, whenever $x\lesssim_{p_1,\dots,p_n} y$ and $y\lesssim_{p_1,\dots,p_n}x$.

\smallskip

Below, $(\O, \mathscr{A},(\mathscr{F}_t)_{t\geq 0}, \P)$ denotes a filtered probability space carrying a sequence of independent standard Brownian motions which changes depending on the SPDE under consideration, and $\E[\cdot]=\int_{\O} \cdot \,\dd \P$ for the associated expected value. 
A process $\phi:[0,\infty)\times \O\to X$ is progressively measurable if $\phi|_{[0,t]\times \O}$ is $\Borel([0,t])\otimes \F_t$ measurable for all $t\geq 0$, where $\Borel$ is the Borel $\sigma$-algebra on $[0,t]$ and $X$ a Banach space. Moreover, a stopping time $\tau$ is a measurable map $\tau:\O\to [0,\infty]$ such that $\{\tau\leq t\}\in \F_t$ for all $t\geq 0$. Finally, a stochastic process $\phi:[0,\tau)\times \O\to X$ is progressively measurable if $\one_{[0,\tau)\times \O}\,\phi$ is progressively measurable where $
[0,\tau)\times \O\stackrel{{\rm def}}{=}\{(t,\om)\in[0,\infty)\times \O\,:\,\,0\leq t<\tau(\om)\}$
and $\one_{[0,\tau)\times \O}$ (or simply $\one_{[0,\tau)}$) stands for the extension by zero outside $[0,\tau)\times \O$. The definition of the stochastic intervals $(0,\tau)\times \O$ and $[0,\tau]\times \O$ is similar.

\smallskip

We write $L^p(S,\mu;X)$ for the Bochner space of strongly measurable, $p$-integrable $X$-valued functions for a measure space $(S,\mu)$, $p\in (1,\infty)$ and a Banach space $X$ as defined in \cite[Section 1.2b]{Analysis1}. Moreover, 
if $X=\R$, we write $L^p(S,\mu)$, and if it is clear which measure we refer to, we also leave out $\mu$. 
Finally, $\one_A$ denotes the indicator function of $A\subseteq S$.
For a Banach space $X$, $p\in (1,\infty)$ and $T\in (0,\infty]$, we denoted by $W^{1,p}(0,T;X)$ the set of all $f\in L^p(0,T;X)$ such that $f'\in L^p(0,T;X)$ endowed with the natural norm, see \cite[Section 2.5]{Analysis1} for distributional derivative of Banach valued maps. Moreover, for $\vartheta\in (0,1)$, we define the Bessel-potential space $H^{\vartheta,p}(0,T;X)$ as
$$
H^{\vartheta,p}(0,T;X)\stackrel{{\rm def}}{=}[L^p(0,T;X),W^{1,p}(0,T;X)]_{\vartheta},
$$
where $[\cdot,\cdot]_{\vartheta}$ denotes the complex interpolation method, see \cite{BeLo,Analysis1} for details. As usual, for an interval $I\subseteq [0,\infty)$, we say $f\in H^{\vartheta,p}_{\loc}(I;X)$ if $f\in H^{\vartheta,p}(J;X)$ for all compact sets $J\subseteq I$ (a similar notation is employed if $H^{\vartheta,p}$ is replaced by either $W^{1,p}$ or $L^p$). 

Next, we introduce the function spaces on $\T^d$, which will be used throughout the manuscript. As above, 
Bessel-potential spaces are indicated by $H^{s,q}(\T^d)$ where $s\in \R$ and $q\in (1,\infty)$. 
We also use the standard short-hand notation $H^{s,2}(\T^d)$ for $H^{s}(\T^d)$. We sometimes also employ Besov spaces $B^{s}_{q,p}(\T^d)$ which can be defined as the real interpolation space $(H^{-k,q}(\T^d),H^{k,q}(\T^d))_{(s+k)/2k,p}$ for $s\in \R$, $\N\ni k>|s|$ and $1<q,p<\infty$. The reader is again referred to \cite{BeLo,Analysis1} for details on interpolation and to \cite[Section 3.5.4]{schmeisser1987topics} for details on function spaces over $\T^d$. Moreover, we set $\mathcal{A}(\T^d;\R^k)\stackrel{{\rm def}}{=}(\mathcal{A}(\T^d))^k$ and $\mathcal{A}(\cdot)\stackrel{{\rm def}}{=}\mathcal{A}(\T^d;\cdot)$ where $\mathcal{A}\in \{L^q,H^{s,q},B^{s}_{q,p}\}$ and $k\in \N$.
Finally, for all $\vartheta_0,\vartheta_1>0$ and $t>0$, we let 
$$
C^{\vartheta_0,\vartheta_1}((0,t)\times \T^d)\stackrel{{\rm def}}{=}C^{\vartheta_0}(0,t;C(\T^d))\cap C([0,t];C^{\vartheta_1}(\T^d))
$$ 
and $C^{\vartheta_0,\vartheta_1}_{\loc}((0,t)\times \T^d)\stackrel{{\rm def}}{=}\cap_{\varepsilon>0}C^{\vartheta_0,\vartheta_1}((\varepsilon,t-\varepsilon)\times \T^d)$.

\section{Statement of the main result and strategy}
\label{s:statement}
This section is organized as follows. In Subsections \ref{ss:helmholtz_related} and \ref{ss:noise}, we collect some facts on function spaces of divergence-free vector fields and illustrate the structure of the noise in the HNSEs \eqref{eq:hyper_NS}, respectively. Finally, in Subsection \ref{ss:global_section}, we state the main result of the manuscript, namely Theorem \ref{t:main_theorem_L2}.

\subsection{Helmholtz projection and related function spaces}
\label{ss:helmholtz_related}
In this subsection, we introduce the Helmholtz projection $\p$ and its complement projection $\q$. For an $\R^d$-valued distribution $f=(f^i)_{i=1}^{d}\in \D'(\T^d;\R^d)$ on $\T^d$, let $\widehat{\,f^i\,}(k) =\langle e_k ,f^i\rangle $ be $k$-th Fourier coefficients, where $i\in \{1,\dots,d\}$, $k=(k_i)_{i=1}^d\in \Z^d$ and $e_k(x)=e^{2\pi \i k\cdot x}$. The Helmholtz projection $\p f$ for $f\in \D'(\T^d;\R^d)$ is given by
\begin{equation}
\label{eq:def_helmholtz_projection_revised}
(\widehat{ \p f})^i (k)\stackrel{{\rm def}}{=}\widehat{\,f^i\,}(k)- \sum_{1\leq j\leq d} \frac{k_i k_j }{|k|^2} \widehat{\,f^j\,}(k), \qquad 
(\widehat{ \p f})^i (0)\stackrel{{\rm def}}{=}\widehat{\,f^i\,}(0).
\end{equation}
Formally, $\p f$ can be written as $f- \nabla \Delta^{-1}(\nabla \cdot f)$. We set 
$$
\q\stackrel{{\rm def}}{=} \mathrm{Id}- \p.
$$  
From standard Fourier analysis, it follows that $\q$ and $\p$ restrict to bounded linear operators on $H^{s,q}(\T^d;\R^d)$ and $B^s_{q,p}(\T^d;\R^d)$ for $s\in \R$ and $q,p\in (1,\infty)$.
Finally, for $s\in \R$ and $q\in (1,\infty)$, we can introduce function spaces of divergence-free vector fields: 
\begin{align*}
\Hs^{s,q}(\T^d)\stackrel{{\rm def}}{=}\p(H^{s,q}(\T^d;\R^d)) \quad \text{ and }\quad 
\Ls^{q}(\T^d)\stackrel{{\rm def}}{=}\p(L^{q}(\T^d;\R^d)) ,
\end{align*}
endowed with the natural norms. As usual, we write $\Hs^{s}(\T^d)$ instead of $\Hs^{s,2}(\T^d)$. Similarly, for $p\in (1,\infty)$, one can define the space $\Bs^{s}_{q,p}(\T^d)$ of Besov spaces divergence vector fields and $\ell^2$-values function spaces.

\subsection{Structure of the noise}
\label{ss:noise}
In this subsection, we describe the structure of the noise we used in \eqref{eq:hyper_NS}. Recall that $\Z_0^3\stackrel{{\rm def}}{=}\Z^3\setminus\{0\}$.  
Through the manuscript, we say that $\theta=(\theta_k)_{k\in \Z^3_0}\in \ell^2$ is \emph{normalized} if $\|\theta\|_{\ell^2(\Z^3_0)}=1$ and \emph{radially symmetric} if 
\begin{equation*}
 \theta_{j}=\theta_k \  \text{ for all $j,k\in\Z_0^3$ \  such that }|j|=|k|.
\end{equation*}
Next, we define the family of vector fields $(\sigma_{k,\alpha})_{k,\alpha}$. 
Here and in the following, we use the shorthand subscript `$k,\alpha$' to indicate $k\in \Z^3_0,\, \alpha\in\{1,2\}$. 

Let $\Z_{+}^3$ and $\Z_-^3$ be a  partition of $\Z_0^3$ such that $-\Z_+^3=\Z_-^3$. For any $k\in \Z_+^3$, let $\{a_{k,1},a_{k,2}\}$ be a complete orthonormal basis of the hyperplane $k^{\bot}=\{k'\in \R^3\,:\, k\cdot k'=0\}$, 
and set $a_{k,\alpha}\stackrel{{\rm def}}{=}a_{-k,\alpha} $ for $k\in \Z^3_-$. Finally, we let
\begin{equation*}
\sigma_{k,\alpha}\stackrel{{\rm def}}{=} a_{k,\alpha} e^{2\pi \i k\cdot x}  \ \  \text{ for all } \ \ x\in \Tor^3,\ k\in \Z^3_0,\ \alpha\in \{1,2\}.
\end{equation*}
By construction, the vector field $\sigma_{k,\alpha}$ is smooth and divergence-free for all $k,\alpha$.

Finally, we introduce the family of complex Brownian motions $(W^{k,\alpha})_{k,\alpha}$. 
Let $(B^{k,\alpha})_{k,\alpha}$ be a family of independent standard (real) Brownian motions on the above-mentioned filtered probability space $(\O,\A,(\F_t)_{t\geq 0},\P)$. Then, we set
\begin{equation}
\label{eq:complex_BM}
W^{k,\alpha}\stackrel{{\rm def}}{=}
\left\{
\begin{aligned}
&B^{k,\alpha}+ \i \,B^{-k,\alpha},& \qquad& k\in \Z^3_+,\\
&B^{-k,\alpha}- \i \, B^{k,\alpha}, &\qquad& k\in \Z^3_-.
\end{aligned}
\right.
\end{equation}
In particular $\overline{W^{k,\alpha}}= W^{-k,\alpha}$ for all $k,\alpha$.

Next, let us transform the Stratonovich formulation of the noise in \eqref{eq:hyper_NS} into an It\^o-one.
As in \cite[Section 2.3]{FL19}, due to the divergence-free of $\sigma_{k,\alpha}$, at least formally, for all radially symmetric $\theta\in \ell^2$, 
\begin{align*}
\sqrt{\frac{3\mu}{2}}
\sum_{k,\alpha}\theta_k\, \p [ (\sigma_{k,\alpha}\cdot \nabla) u]\circ \dot{W}_t^{k,\alpha}
&=
\LL u+\sqrt{\frac{3\mu}{2}}
\sum_{k,\alpha}\theta_k \,\p [ (\sigma_{k,\alpha}\cdot \nabla) u]\, \dot{W}_t^{k,\alpha},
\end{align*}
where 
\begin{equation}
\label{eq:Ito_stratonovich_change}
\LL u\stackrel{{\rm def}}{=}
\frac{3\mu}{2}
\sum_{k,\alpha}\theta_k^2 \,\p \big[ \nabla \cdot ( \p[(\sigma_{k,\alpha}\cdot \nabla)u]\otimes \sigma_{-k,\alpha}\big)\big].
\end{equation}
Here, as usual, $\nabla \cdot A= (\partial_j A_{i,j})_{j=1}^3$ for a matrix valued map $A=(A_{i,j})_{i,j=1}^3$.

Integrating by parts, one can check that $\LL $ is a negative operator for all $\theta\in \ell^2$, i.e., $\langle \LL v,v \rangle_{H^{-1},H^1}\leq 0$ for all $v\in \Hs^1(\T^3)$. However, let us stress that the operator $\LL$ does \emph{not} create any additional dissipation, as the dissipative effect of $\LL $ is balanced by the energy creation of the transport noise in It\^o-form, cf.\ \eqref{eq:energy_estimate_unchanged}.

\subsection{Global existence of smooth solutions of HNSEs by transport noise}
\label{ss:global_section}
Before stating our main result, let us introduce the notion of solutions to \eqref{eq:hyper_NS}. 
We begin by recalling that the sequence of complex Brownian motions $(W^{k,\alpha})_{k,\alpha}$ induces an $\ell^2$-cylindrical Brownian motion $\mathcal{W}_{\ell^2}$ via the formula (cf.\ e.g., \cite[Example 2.12]{AV19_QSEE1})
\begin{equation}
\label{eq:cylindrical}
\mathcal{W}_{\ell^2}(f)=\sum_{k,\alpha}\int_{\R_+} f_{k,\alpha}(t)\,\dd W^{k,\alpha}_t 
\ \ \  \text{ for } \ \ \
f=(f_{k,\alpha})_{k,\alpha}\in L^2(\R_+;\ell^2).
\end{equation}
Note that $
\mathcal{W}_{\ell^2}(f)$ is real-valued if $f_{k,\alpha}=f_{-k,\alpha}$ as $\overline{W^{k,\alpha}}= W^{-k,\alpha}$ for all $k,\alpha$. Using this and the symmetry of $\sigma_{k,\alpha}$ under the reflection $k\mapsto -k$, one can check that the stochastic perturbation in \eqref{eq:hyper_NS} can be rewritten only using real-valued coefficients and the real-valued Brownian motions $(B^{k,\alpha})_{k,\alpha}$ in \eqref{eq:complex_BM}, cf.\ \cite[Subsection 2.3]{A24_anomalous} and Subsection \ref{ss:proofs_local} below.
Hence, solutions to \eqref{eq:hyper_NS} are naturally real-valued. 
Finally, as usual, to accommodate the weak analytic setting, we rewrite the convective terms $(u\cdot\nabla)u$ as $\nabla \cdot (u\otimes u)$ due to the divergence-free condition.

\begin{definition}[Local, unique and maximal $p$-solutions to HNSEs]
\label{def:p_solution} 
Assume that $u_0\in L^0_{\F_0}(\O;\Ls^2(\T^3))$ and $\theta=(\theta_k)_{k\in \Z^3_0}\in \ell^2$ is radially symmetric, and fix $p\in [2,\infty)$. Let $\tau:\O\to [0,\infty]$ and $u:[0,\tau)\times \O\to \Hs^\g(\T^3) $ 
be a stopping time and a progressive measurable process, respectively.

\begin{itemize}
\item We say that $(u,\tau)$ is a \emph{local $p$-solution} to \eqref{eq:hyper_NS} if there exists a sequence of stopping times $(\tau_n)_{n\geq 1}$ such that $ \tau_n\uparrow \tau$ a.s.\ for which the following are satisfied for all $n\geq 1$:
\begin{enumerate}[{\rm(a)}]
\vspace{0.1cm}
\item\label{it:integrability_1} $u\in L^p(0,\tau_n;\Hs^\g(\T^3))$ a.s.;
\vspace{0.1cm}
\item\label{it:integrability_2} $\nabla \cdot (u\otimes u) \in L^p(0,\tau_n;H^{-\g}(\T^3;\R^{3}))$ a.s.;
\vspace{0.1cm}
\item\label{it:integrability_3} a.s.\ for all $t\in [0,\tau_n]$ it holds that 
\begin{align*}
u(t)-u_0
= &\int_0^t \Big(\LL u(s) -(-\Delta)^\g u(s) - \p[\nabla \cdot (u(s)\otimes u(s))]\Big) \,\dd s\\
+ &\, \sqrt{\frac{3\mu}{2}}
\int_0^t \one_{[0,\tau_n]}\Big(\p[(\theta_k\,\sigma_{k,\alpha}\cdot\nabla)u(s)\big]\Big)_ {k,\alpha}\,\dd 
\mathcal{W}_{\ell^2}.
\end{align*}
\end{enumerate}
\item A local $p$-solution $(u,\tau)$ to \eqref{eq:hyper_NS} is said to be a \emph{unique (pathwise) local $p$-solution} to \eqref{eq:hyper_NS} if for any local solution $(u',\tau')$ we have $u'=u$ a.e.\ on $[0,\tau'\wedge \tau)\times \O$. 
\item A unique local $p$-solution $(u,\tau)$ to \eqref{eq:hyper_NS} is said to be a \emph{unique maximal $p$-solution} to \eqref{eq:hyper_NS} if for any local solution $(u',\tau')$ we have $\tau'\leq \tau$ a.s.\ and $u'=u$ a.e.\ on $[0,\tau')\times \O$. 
\end{itemize}
\end{definition}

Due to \eqref{it:integrability_1}-\eqref{it:integrability_2} and $\theta\in \ell^2$, the deterministic and stochastic integrals in \eqref{it:integrability_3} are well-defined as $\Hs^{-\g}(\T^3)$-valued Bochner and $\Ls^2(\T^3)$-valued It\^o integrals, respectively. Thus, the equality in \eqref{it:integrability_3} is understood as an identity in $\Hs^{-\g}(\T^3)$.

We emphasize that the integrability condition in \eqref{it:integrability_2} is a requirement for a progressively measurable process $u$ to be a local $p$-solution, and it does not follow from \eqref{it:integrability_1}. However, from the fact that the process $u$ takes values in $\Hs^\g(\T^3)$ a.e.\ on $[0,\tau)\times \O$ and Lemma \ref{lem:nonlinearity_estimate}, it follows that $\nabla \cdot (u\otimes u)\in H^{-\g}(\T^3;\R^3)$ a.e.\ on $[0,\tau)\times \O$. Hence, the integrability demand \eqref{it:integrability_2} is well-defined as it concerns the a.s.\ integrability of the mapping $[0,\tau_n]\ni t\mapsto \|\nabla \cdot (u(t,\cdot)\otimes u(t,\cdot))\|^p_{H^{-\g}(\T^3;\R^{3})}$ for $n\geq 1$, cf.\ \cite[Proposition 1.2.2]{Analysis1}.
Alternatively, under the additional  assumptions $p\geq \frac{4\g}{6\g-5}$ and $u_0\in \Bs^{\g(1-2/p)}_{2,p}(\T^3)$, which will be enforced throughout this manuscript (see also the text below \eqref{eq:smr_p_intro}-\eqref{eq:convergence_CB} for comments), one can replace the condition \eqref{it:integrability_2} in Definition \ref{def:p_solution} by
\begin{itemize}
\item
$u\in C([0,\tau_n];\Bs^{\g(1-2/p)}_{2,p}(\T^3))$ a.s.
\end{itemize} 
Indeed, on the one hand, for $\beta$ as in Lemma \ref{lem:nonlinearity_estimate}, one can check that interpolating \eqref{it:integrability_1} and the above item, one obtains $u\in L^{2p}(0,\tau_n;\Hs^{\beta}(\T^3))$ a.s. The latter implies \eqref{it:integrability_2} by Lemma \ref{lem:nonlinearity_estimate} again. On the other hand, the conditions
\eqref{it:integrability_1}-\eqref{it:integrability_2}, the stochastic maximal $L^p$-regularity estimate of Theorem \ref{t:smr_theta} for the linearized problem of \eqref{eq:hyper_NS} below, and a localization argument (see e.g., \cite[Proposition 3.11]{AV25_survey}) imply 
$u\in C([0,\tau_n];\Bs^{\g(1-2/p)}_{2,p}(\T^3))$ a.s. 
Here, we prefer the formulation as given in Definition \ref{def:p_solution} as it does not directly involve Besov spaces.

\smallskip

The main result of the current manuscript reads as follows.

\begin{theorem}[Global smooth solutions by transport noise of HNSEs]
\label{t:main_theorem_L2}
Let $1<\g<\frac{5}{4}$. Fix $\frac{4\g}{6\g-5}<p<\infty$,
$$
\varepsilon\in (0,1)\quad \text{ and }\quad N\geq 1.
$$
Then there exist $\mu>0$ and $\theta\in \ell^2$ such that $\#\,\supp\theta<\infty$ for which the following assertion holds. 
For all initial data $u_0$ satisfying 
$$
u_0\in \Bs^{\g(1-2/p)}_{2,p}(\T^3)\qquad \text{ and }\qquad \|u_0\|_{B^{\g(1-2/p)}_{2,p}(\T^3;\R^3)}\leq N,
$$
there exists a \emph{unique} maximal $p$-solution $(u,\tau)$ to the stochastic {\normalfont{HNSEs}} \eqref{eq:hyper_NS} that is \emph{global in time} with high probability:
\begin{equation}
\label{eq:global_high_probability}
\P(\tau=\infty)>1-\varepsilon.
\end{equation}
Moreover, $(u,\tau)$ is regular up to the initial time and instantaneously gains regularity in time and space:
\begin{align}
\label{eq:main_theorem_L2_regularity_1}
u&\in H^{\vartheta,p}_{\loc}([0,\tau);H^{\g(1-2\vartheta)}(\T^3;\R^3))\ \text{ a.s.\ for all }\vartheta\in [0,1/2),\\
\label{eq:main_theorem_L2_regularity_2}
u&\in C([0,\tau);B^{\g(1-2/p)}_{2,p}(\T^3;\R^3))\  \text{ a.s.},\\
\label{eq:main_theorem_L2_regularity_3}
u&\in C^{\vartheta_0,\vartheta_1}_{\loc}((0,\tau)\times \T^3;\R^3)\ \text{ a.s.\ for all }\vartheta_0\in [0,1/2), \, \vartheta_1<\infty.
\end{align}
\end{theorem}

The proof of the above is given in Subsection \ref{ss:global_smooth_proof}.  
The main difficulty behind Theorem \ref{t:main_theorem_L2} is to obtain \eqref{eq:global_high_probability} for unique $p$-solutions to \eqref{eq:hyper_NS}. Indeed, the local well-posedness and instantaneous regularization of solutions to the stochastic HNSEs \eqref{eq:hyper_NS} readily follow from the theory of stochastic evolution equations in critical spaces, see \cite{AV19_QSEE1,AV19_QSEE2} or \cite{AV25_survey}. Some details are given in Section \ref{s:local_well_stoch_problem}.
The global well-posedness with high probability in \eqref{eq:global_high_probability} is a consequence of the scaling limit result of Theorem \ref{t:scaling_limit_HNSE}. The proof of the latter result is given at the end of Section \ref{s:scaling_limit}.

Note that, for all $\frac{5}{2}-2\g<\delta<\g$ and $\frac{4\g}{6\g-5}<p<\frac{2\g}{\g-\delta}$, it holds that 
$$
H^{\delta}(\T^3;\R^3)\embed B^{\g(1-2/p)}_{2,p}(\T^3;\R^3).
$$
Therefore, the result of Theorem \ref{t:main_theorem_L2} also applies to initial data in $H^{\delta}$ with $\delta>\frac{5}{2}-2\g$. In this case, \eqref{eq:main_theorem_L2_regularity_2} can be strengthened to $u\in C([0,\tau);\Hs^\delta(\T^3))$ a.s. Although it is possible to state Theorem \ref{t:main_theorem_L2} solely in terms of $H^\delta$-data, we prefer to present it in terms of initial data in the space $B^{\g(1-2/p)}_{2,p}$, as the proof of Theorem \ref{t:main_theorem_L2} essentially uses on $L^p$-techniques (see Subsection \ref{ss:novelty}), where Besov spaces naturally arise.

We conclude this section by gathering some observations in the following

\begin{remark}\
\label{r:generalization}
\begin{itemize}
\item 
By using $L^q$-theory in the spatial variable with $q\neq p$, one can extend Theorem \ref{t:main_theorem_L2} to \emph{rougher} initial data, e.g., $u_0\in \Ls^q(\T^3)$ with $q>\frac{3}{2\g-1}$ (note that the integrability $\frac{3}{2\g-1}$ is critical for the HNSEs, cf.\ Subsection \ref{ss:novelty}).  
\item 
The choice of $\theta$ in Theorem \ref{t:main_theorem_L2} is not unique. Indeed, as in \cite{FL19}, for all $\alpha>0$, the proof of the latter shows that there exists $n_0$ depending only on $\g,p,\varepsilon,N$ and $\alpha$ such that \eqref{eq:global_high_probability}-\eqref{eq:main_theorem_L2_regularity_3} hold with $\theta=\theta^n$ for all $n\geq n_0$ where
$$
\theta^n =\frac{\Theta^n}{\|\Theta^n\|_{\ell^2}}, \qquad  \Theta^n_k =\frac{1}{|k|^{\alpha}}\,\one_{\{n \leq |k|\leq 2n\} } \ \text{ for }\ k\in \Z^3_0.
$$
\item Theorem \ref{t:main_theorem_L2} also holds if $(-\Delta)^{\g}$ is replaced by $\nu(-\Delta)^{\g}$ with $\nu>0$. In the latter case, the choice of $\theta$ depends on $\nu>0$ (see \cite[Subsection 1.2]{A24_anomalous} for physical motivation of $\nu$-dependent transport noise).
\end{itemize}
\end{remark}

\section{Maximal $L^p$-estimates for HSEs with $L^\infty$-transport noise}
\label{app:smr}
In this section, we prove the main linear estimate needed in the proof of Theorem \ref{t:main_theorem_L2}. More precisely, we establish maximal $L^p$-regularity estimates for the hyperviscous Stokes equations (HSEs in the following) with $L^{\infty}$-transport noise:  
\begin{equation}
\label{eq:hyper_NS_linear}
\left\{
\begin{aligned}
\partial_t v
+ (-\Delta)^{\g} v 
 &= (w\cdot\nabla) v+ \LLxi_{\xi}(v)+ f + \sum_{n\geq 1}\big( \p [ (\xi_{n}\cdot\nabla) v]+ g_{n} \big)\, \dot{W}^{n}&\text{on }&\T^d,
\\
v(s,\cdot)&=v_s&\text{on }&\T^d,
\end{aligned}
\right.
\end{equation}
where $s\geq 0$, $w\in \R^d$, $(\xi_n)_{n\geq 1}$ is an $\ell^2$-bounded family of vector fields, $(W^n)_{n\geq 1}$ is a sequence of standard independent Brownian motions on a filtered probability space (which might differ from the one used in the main body of this work, but is still denoted by $(\O,\mathscr{A},(\F_t)_{t\geq 0},\P)$), 
\begin{equation}
\LLxi_\xi (v) \stackrel{{\rm def}}{=}\frac{1}{2}\sum_{n\geq 1} \p\Big[\nabla \cdot\big( \p[(\xi_n\cdot\nabla)v]\otimes \xi_n\big)\Big],
\end{equation}
and $\p$ is the Helmholtz projection, see \eqref{eq:def_helmholtz_projection_revised}.
As in Subsection \ref{ss:noise}, if $\nabla\cdot \xi_n=0$ in $\D'(\T^d)$ and $g_n\equiv 0$ for all $n\geq 1$, then the problem \eqref{eq:hyper_NS_linear} is the  (formal) It\^o reformulation of the Stratonovich SPDE:
$$
\partial_t v
+ (-\Delta)^{\g} v 
 = (w\cdot\nabla) v+  f + \sum_{n\geq 1} \p [ (\xi_{n}\cdot\nabla) v]\circ \dot{W}^{n} \ \ \text{ on }\T^d,
$$
In particular, the problem \eqref{eq:hyper_NS_linear} is the linearization of the HSNEs \eqref{eq:hyper_NS}.
The additional transport term $(w\cdot\nabla)v$ is useful in handling initial data with non-zero mean, see also the proof of Theorem \ref{t:small_implies_global}. 

\smallskip

Throughout this section, we assume that $\xi_n :\R_+\times\O\times \T^d\to \R^d$ is progressively measurable for all $n\geq 1$, and there exists a constant $N\geq 1$ for which 
\begin{equation}
\label{eq:boundedness_xi}
\|(\xi_n)_{n\geq 1}\|_{\ell^2}\leq N \ \ \text{ a.e.\ on }\ \R_+\times \O\times \T^d.
\end{equation}
We now define solutions to \eqref{eq:hyper_NS_linear} similarly to the one of \eqref{eq:hyper_NS} given in Definition \ref{def:p_solution}.
Fix $s\geq 0$, $v_s\in L^0_{\F_s}(\O;\Ls^2(\T^d))$ and a stopping time $\tau:\O\to [s,\infty]$.
A progressively measurable process $v:[s,\tau]\times \O\to \Hs^\g(\T^d)$ is said to be a \emph{$p$-solution to \eqref{eq:hyper_NS_linear} on $(s,\tau)$} provided $v\in L^p(s,\tau;\Hs^{\g}(\T^d))$ a.s.\ and, 
a.s.\ for all $t\in (s,\tau)$,
$$
v(t)=v_s+\int_{s}^t \big[-(-\Delta)^\g v+ (w\cdot\nabla)v +\LLxi_{\xi} v \big]\,\dd r 
+ \int_s^t   \one_{[s,\tau]}\big(\p [ (\xi_{n}\cdot\nabla) v]+ g_n\big)_{n\geq 1}\,\dd \mathcal{W}_{\ell^2}.
$$
Here $\mathcal{W}_{\ell^2}$ is the $\ell^2$-cylindrical Brownian motion induced by $(W^n)_{n\geq 1}$, see \eqref{eq:cylindrical}.

\smallskip

Maximal $L^p$-regularity estimates for SPDEs were first studied by Krylov (see e.g., \cite{Kry94a,Kry}) and later partially extended via the semigroup approach by van Neerven, Veraar, and Weis \cite{MaximalLpregularity}. 
The reader is referred to \cite[Subsection 1.2]{AV19_QSEE1} and \cite[Section 3]{AV25_survey} for brief overviews of known results.

Here, we prove maximal $L^p$-regularity estimates for the stochastic HSEs \eqref{eq:hyper_NS_linear} by combining the results in \cite{MaximalLpregularity} and the perturbation result in \cite[Section 3]{AV_torus}.

\begin{theorem}[Maximal $L^p$-regularity -- HSEs with $L^\infty$-transport noise]
\label{t:smr_theta}
Assume that $\g>1$ and $p\in [2,\infty)$. Let $(\xi_n)_{n\geq 1}$ be as above and let \eqref{eq:boundedness_xi} be satisfied for some deterministic constant $N\geq 1$.
Suppose that $\nabla \cdot \xi_n=0$ in $\D'(\T^d)$ for all $n\geq 1$, and $w\in \R^d$ satisfies $|w|\leq N$.
Fix $ s\geq 0$, $t\in (s,\infty]$, $v_s\in L^p_{\F_s}(\O;\Bs^{\g(1-2/p)}_{2,p}(\T^d))$, a stopping time $\tau:\O\to [s,t]$, and progressively measurable processes
\begin{equation*}
f\in L^p( (s,\tau)\times \O;\Hs^{-\g}(\T^d)),
\quad \  g=(g_{n})_{n\geq 1}\in L^p( (s,\tau)\times \O;\Ls^{2}(\T^d;\ell^2)).
\end{equation*}  
Then there exists a unique global $p$-solution $v$ to \eqref{eq:hyper_NS_linear} on $(s,\tau)$ for which the following hold:
\begin{enumerate}[{\rm(1)}]
\item\label{it:smr_theta_1} If $t<\infty$, then there exists a constant $C_1(\g,N,p,t-s)>0$ such that
\begin{align*}
&\E\Big[ \sup_{r\in [s,\tau]}\|v(r)\|_{B^{\g(1-2/p)}_{2,p}(\T^d;\R^d)}^p\Big]
+\E\int_s^\tau \|v(r)\|_{H^{\g}(\T^d;\R^d)}^p \,\dd r\\
&\quad
\leq C_1 \E\|v_s\|_{B^{\g(1-2/p)}_{2,p}(\T^d;\R^d)}^p
+ C_1  \E\int_s^\tau \big(\|f(r)\|_{H^{-\g}(\T^d;\R^d)}^p + \|g(r)\|_{L^{2}(\T^d;\ell^2)}^p\big) \,\dd r.
\end{align*}
\item\label{it:smr_theta_2} 
If 
$
\int_{\T^d} v_s\,\dd x =0 $ a.s.\ and 
$ 
\int_{\T^d} f\,\dd x 
= 
\int_{\T^d} g_n\,\dd x=0
$
a.e.\ on $(s,\tau) \times \O$ for all $n\geq 1$ and $v\in L^p( (s,\tau)\times \O;L^2(\T^d;\R^d))$, then there exists $C_2(\g,N,p)>0$ such that
\begin{align*}
&\E \Big[\sup_{r\in [s,\tau]}\|v(r)\|_{B^{\g(1-2/p)}_{2,p}(\T^d;\R^d)}^p\Big]
+
\E\int_s^{\tau} \|v(r)\|_{H^{\g}(\T^d;\R^d)}^p \,\dd t\\
&\qquad\qquad
\leq C_2 \E\int_s^\tau \big(\|f(r)\|_{H^{-\g}(\T^d;\R^d)}^p + \|g(r)\|_{L^{2}(\T^d;\ell^2)}^p\big) \,\dd r\\
&\qquad\qquad
+C_2 \E\|v_s\|_{B^{\g(1-2/p)}_{2,p}(\T^d;\R^d)}^p+C_2\E\int_s^\tau \|v(r)\|_{L^2(\T^d;\R^d)}^p \,\dd r .
\end{align*}
\end{enumerate}
\end{theorem}

As the proof below shows, the above also holds with $L^q$-integrability in $x\in \T^d$, and with additional optimal space-time regularity estimates in the spaces 
$$
L^p(\O;H^{\vartheta,p}(0,\tau;H^{\g(1-2\vartheta)}(\T^d;\R^d))) \ \text{ for } \  \vartheta\in [0,1/2),
$$ 
cf.\ \cite[Definition 3.5]{AV19_QSEE1} or \cite[Definition 3.8]{AV25_survey}. However, the latter are not needed in this manuscript. Moreover, we emphasize that the constants $C_1$ and $C_2$ are independent of $s\geq 0$.

By comparing \eqref{it:smr_theta_1} with \eqref{it:smr_theta_2}, we observe that in the latter, $t = \tau = \infty$ is allowed, but at the cost of introducing a term that depends on the $L^2$-norm of $v$, which is lower-order with respect to the estimated $H^\g$-norm.
As shown in Section \ref{s:scaling_limit}, the key aspect of Theorem \ref{t:smr_theta} is that the constants $C_1$ and $C_2$ depend on $(\xi_n)_{n \geq 1}$ only through $N \geq 1$, which controls its $L^\infty(\R_+ \times \O \times \T^d; \ell^2)$-norm. Importantly, no smoothness assumption is required for $(\xi_n)_{n \geq 1}$, as discussed in Subsection \ref{ss:novelty}.

\begin{proof}
For brevity, below we only prove \eqref{it:smr_theta_2}.
The proof of \eqref{it:smr_theta_1} is similar and simpler.
In particular, we assume that $v_s,f$ and $g_n$ have mean zero on $\T^d$.
The main idea in proving the estimate in \eqref{it:smr_theta_2} is to look at \eqref{eq:hyper_NS_linear} as a perturbation of the case $\xi_n\equiv 0$ as the transport noise is lower order compared to $(-\Delta)^{\g}$ as $\g>1$, cf.\ \cite[Theorem 3.2]{AV_torus}. 
To this end, let us first discuss the case where $\xi_n\equiv 0$. For exposition convenience, we also assume $w=0$ in Step 1. However, the latter is not essential.
Below, we write $L^2$ instead of $L^2(\T^d;\R^d)$ etc.\ if no confusion seems likely.

\smallskip

\emph{Step 1: If $\xi_n\equiv 0$ and $w=0$, then for all stopping time $\tau:\O\to [s,\infty]$ and mean-zero processes $f,g_n$ and initial data $v_s$ as in the statement of Theorem \ref{t:smr_theta}, there exists a unique $p$-solution $v$ to \eqref{eq:hyper_NS_linear} on $(s,\tau)$ satisfying}
\begin{align*}
&\E \Big[\sup_{r\in [s,\tau]}\|v(r)\|_{B^{\g(1-2/p)}_{2,p}(\T^d;\R^d)}^p\Big]+
\E\int_s^\tau \|v(r)\|_{H^{\g}(\T^d;\R^d)}^p \,\dd r \\
\nonumber
&\quad
\leq C_0 \E\|v_s\|_{B^{\g(1-2/p)}_{2,p}(\T^d;\R^d)}^p+  C_0\E\int_s^{\tau} \big(\|f(r)\|_{H^{-\g}(\T^d;\R^d)}^p + \|g(r)\|_{L^{2}(\T^d;\ell^2)}^p\big) \,\dd r,
\end{align*}
\emph{where $C_0$ depends only on $p,\g$ and $d$ (in particular, $C_0$ is independent of $s\geq 0$)}.

As $\xi_n\equiv 0$, here we can apply the semigroup approach to stochastic maximal $L^p$-regularity estimates.
Due to \cite[Proposition 3.12(b)]{AV19_QSEE2}, it is enough to show the above estimate with $\tau$ replaced by $t$.
For $\sigma\in\R$, let us set 
$$
\textstyle
\Hsd^{\sigma}(\T^d)\stackrel{{\rm def}}{=}\{w\in
\Hsd^{\sigma}(\T^d) \,:\,\int_{\T^d}w(x)\, \dd x =0\},
$$
where $\int_{\T^d}w(x)\, \dd x \stackrel{{\rm def}}{=}\langle 1,w\rangle$ if $\sigma<0$. Note that the operator
\begin{equation}
\label{eq:def_Aq}
\begin{aligned}
A: 
\Hsd^{\g}(\T^d)
\subseteq 
\Hsd^{-\g}(\T^d)&\to 
\Hsd^{-\g}(\T^d),\\
v&\mapsto (-\Delta)^{\g}v,
\end{aligned}
\end{equation}
is invertible and the periodic version of \cite[Theorem 10.2.25]{Analysis2} implies that $A$ has a bounded $H^{\infty}$-calculus with angle $0$ (see \cite[Chapter 10]{Analysis2}). Hence, the claim follows from \cite[Theorem 1.2]{MaximalLpregularity}. Let us note that the independence of the constant $C_0$ on $s\geq 0$ can be seen by the deterministic characterisation of the stochastic maximal $L^p$-regularity for semigroup generator in \cite[Proposition 3.7]{AV19}.

\smallskip

\emph{Step 2: Conclusion -- Proof of \eqref{it:smr_theta_2}}.
Recall that $u_s,f$, and $g_n$ have mean zero. Thus, any $p$-solution to the HSEs \eqref{eq:hyper_NS_linear} satisfy $\int_{\T^d} v(t,x)\,\dd x =0$ a.s.\ on for all $t\in (s,\tau)$ since $\int_{\T^d} (\xi_n\cdot\nabla) v\,\dd x =0$ for all $v\in H^{\g}$ as $\nabla\cdot\xi_n=0$ in $\D'(\T^d)$ by assumption.
As mentioned at the beginning of the proof, we prove the claimed estimate by applying \cite[Theorem 3.2]{AV_torus} with the choice $X_1=\Hs^{\g}(\T^d)$, $X_0=\Hs^{-\g}(\T^d)$, $H=\ell^2(\N)$, $\kappa=0$ (no weights in time), $A$ as in \eqref{eq:def_Aq}, $B=0$, and for $v\in X_1$,
$$
A_0 v= -(w\cdot\nabla) v-
\LLxi_{\xi} (v), \qquad B_0=
(\p[(\xi_n\cdot\nabla)v])_{n\geq 1}. 
$$ 
We claim that, for all $\varepsilon>0$ there exists $C_\varepsilon>0$ depending only on $\g,N,d$ and $\varepsilon$ such that, for all $v\in X_1$,
\begin{equation}
\label{eq:claimed_interpolation_inequality}
\|A_0v\|_{X_0} +\|B_0 v\|_{L^2(\T^d;\ell^2)}\leq \varepsilon\|v\|_{X_1}+ C_\varepsilon \|v\|_{L^2}.
\end{equation}
To see the correspondence of \eqref{eq:claimed_interpolation_inequality} with the conditions \cite[Theorem 3.2]{AV_torus}, let us recall that $X_{1/2}=[X_0,X_1]_{1/2}=L^{2}(\T^d)$ and therefore $\calL_2(\ell^2,X_{1/2})=L^2(\T^d;\ell^2)$ by Fubini's theorem.

Let us first show how \eqref{eq:claimed_interpolation_inequality} yields \eqref{it:smr_theta_2}. Let $\tau:\O\to [s,\infty]$ be a stopping time.
The existence of a unique $p$-solution on $(s,\tau)$ 
to \eqref{eq:hyper_NS} such that
\begin{equation}
\label{eq:smr_regularity_appendix}
v\in L^p(\O;L^p ((s,\tau\wedge t);\Hs^{\g}(\T^d))\cap C([s,\tau\wedge t];\Bs^{\g(1-2/p)}_{2,p}(\T^d)))
\end{equation} 
for all $0\leq s<t<\infty$,
follows from \cite[Theorem 3.2]{AV_torus}, Step 1 and \eqref{eq:claimed_interpolation_inequality} with $\varepsilon$ sufficiently small depending only on $C_0$ in Step 1. 
As the existence and uniqueness are proven, the estimate in item \eqref{it:smr_theta_2} follows by observing that Step 1 and \eqref{eq:claimed_interpolation_inequality} imply, for all $t\in (s,\infty)$,
\begin{align*}
&\E\Big[ \sup_{r\in [s,\tau\wedge t]}\|v(r)\|_{B^{\g(1-2/p)}_{2,p}}^p\Big]+
\E\int_s^{\tau\wedge t} \|v(r)\|_{H^{\g}}^p \,\dd r \\
&
\leq C_0 \E\|v_s\|_{B^{\g(1-2/p)}_{2,p}}^p+  C_0\E\int_s^{\tau\wedge t} \big(\|f(r)-A_0 v(r)\|_{H^{-\g}}^p 
+ \|g(r)+B_0v(r)\|_{L^{2}(\ell^2)}^p\big) \,\dd r\\
&
\stackrel{(i)}{\leq} C_1 \E\|v_s\|_{B^{\g(1-2/p)}_{2,p}}^p+  C_1\E\int_s^{\tau\wedge t}\big(\|f(r)\|_{H^{-\g}}^p 
+ \|g(r)\|_{L^{2}(\ell^2)}^p\big) \,\dd r\\
&+ C_1 \E\int_s^{\tau\wedge t} \|v(r)\|_{L^2}^p \,\dd r + \frac{1}{2}\E\int_s^{\tau\wedge t} \|v(r)\|_{H^{\g}}^p \,\dd r  ,
\end{align*} 
where $C_1(\g,N,d,C_0)>0$ and in $(i)$ we applied \eqref{eq:claimed_interpolation_inequality} with $\varepsilon$ depending only on $C_0$. Due to \eqref{eq:smr_regularity_appendix}, the last term on the RHS of the above can be absorbed on the LHS of the corresponding bound. Hence, the claimed estimate in \eqref{it:smr_theta_2} follows by letting $t\uparrow \infty$ and Fatou's lemma. 

It remains to prove \eqref{eq:claimed_interpolation_inequality}.
By \eqref{eq:boundedness_xi} and standard interpolation inequalities, one has, for all $v\in H^{\g}$ and $\varepsilon>0$,
\begin{align}
\label{eq:smallness_Pxi}
\|\LLxi_{\xi} (v)\|_{H^{-\g}}
\leq C_{\g,d} \|\LLxi_{\xi} (v)\|_{H^{-1}}&\leq C_{\g,d} N^2 \|v\|_{H^{1}} \\
\nonumber
& \leq \varepsilon \|v\|_{H^{\g}}+ C_{\g,d,N,\varepsilon} \|v\|_{L^2}.
\end{align}
Similarly, since $\g>1$, we have $
\|(w\cdot\nabla) v\|_{H^{-\g}}
\lesssim_{\g,d}N \|v\|_{L^2}$ and
\begin{align}
\label{eq:smallness_Pxi2}
\|((\xi_n\cdot\nabla)v)_{n\geq 1}\|_{L^2(\ell^2)}
\lesssim N \|v\|_{H^{1}}\leq 
 \varepsilon \|v\|_{H^{\g}}+  C_{\g,d,N,\varepsilon}\|v\|_{L^2}.
\end{align}
The boundedness of $\p$ on $H^{\varrho}$ for $\varrho\in\R$, and \eqref{eq:smallness_Pxi}-\eqref{eq:smallness_Pxi2} yield \eqref{eq:claimed_interpolation_inequality}.
\end{proof}

\section{Local well-posedness and regularity for the stochastic HNSEs}
\label{s:local_well_stoch_problem}
In this section, we analyze the local well-posedness and regularity of stochastic HNSEs in an $L^p(L^2)$ setting. As noted in the comments below Theorem \ref{t:main_theorem_L2} (see also Remark \ref{r:generalization}), the $L^2$-setting in space is chosen for simplicity, while the $L^p$-setting in time is essential.%

For $s\geq 0$, consider
\begin{equation}
\label{eq:hyper_NS_s}
\left\{
\begin{aligned}
\partial_t u + \p[\nabla \cdot (u\otimes u)] 
&+ (-\Delta)^{\g} u  \\
& = \sum_{k\in \Z^3_0}\sum_{\alpha\in \{1,2\}} \theta_k\, \p\big[(\sigma_{k,\alpha}\cdot\nabla) u \big]\circ \dot{W}^{k,\alpha}& \text{ on }&\T^3,\\
\nabla\cdot u&=0&\text{ on }&\T^3,\\
u(s,\cdot)&=u_s&\text{ on }&\T^3.
\end{aligned}
\right.
\end{equation} 
The flexibility in the initial time $s\geq 0$ will be used in the proof of Theorem \ref{t:main_theorem_L2}. Local, unique and maximal $p$-solutions to \eqref{eq:hyper_NS_s} are defined as in Definition \ref{def:p_solution}, replacing the initial time $t=0$ by $t=s$.

\smallskip

Now, the discussion splits naturally between the case of small hyperviscosity $1<\g<\frac{5}{4}$ and the one of large hyperviscosity $\g\geq \frac{5}{4}$. We begin with the case of small hyperviscosity as it is the main focus of the present manuscript. 

\begin{theorem}[Local well-posedness --  HNSEs with transport noise and small hyperviscosity]
\label{t:local_HNSEs_hyper_small_viscosity}
Let $\theta\in \ell^2$ be radially symmetric, $1<\g<\frac{5}{4}$ and $s\geq 0$. 
Assume that $p<\infty$ satisfies
\begin{equation}
\label{eq:HNSEs_critical}
\textstyle
p\geq \frac{4\g}{6\g-5}.
\end{equation}
Then, for all $u_s\in L^0_{\F_s}(\O;\Bs^{\g(1-2/p)}_{2,p}(\T^3))$, there exists a unique local $p$-solution $(u,\tau)$ to \eqref{eq:hyper_NS_s} satisfying $\tau>s$ a.s.\ and
\begin{align}
\label{eq:regularity_paths_local_small_1}
u&\in H^{\vartheta,p}_{\loc}([s,\tau);H^{\g(1-2\vartheta)}(\T^3;\R^3))\ \text{ a.s.\ for all }\vartheta\in [0,1/2),\\
\label{eq:regularity_paths_local_small_2}
u&\in C([s,\tau);B_{2,p}^{\g(1-2/p)}(\T^3;\R^3))\  \text{ a.s.}
\end{align}
Moreover, if $\#\,\supp\theta<\infty$, then the solution $(u,\tau)$ instantaneously regularizes in time and space:
\begin{align}
\label{eq:instantaneous_reg_local}
u&\in C^{\vartheta_0,\vartheta_1}((s,\tau)\times \T^3;\R^3)\ \text{ a.s.\ for all }\vartheta_0\in [0,1/2), \, \vartheta_1<\infty.
\end{align}
\end{theorem}

If $p=\frac{4\g}{6\g-5}$, then $\g(1-2/p)= 5/2-2\g$ and therefore the space of initial data is critical (or scaling invariant) for the HNSEs, see the discussion below \eqref{eq:scaling_HNSEs}. In the remaining cases $p>\frac{4\g}{6\g-5}$, the corresponding space for the initial data is subcritical. 

Next, for completeness, we discuss the case of large hyperviscosity $\g\geq \frac{5}{4}$. Indeed, in such a case, as expected, global well-posedness in the energy space $\Ls^2(\T^3)$ holds.

\begin{theorem}[Global well-posedness -- HNSEs with transport and large hyperviscosity]
\label{t:local_HNSEs_hyper_high_viscosity}
Let $\theta\in \ell^2$ be radially symmetric and let $\g\geq \frac{5}{4}$. 
Then, for all $s\geq 0$ and $u_s\in L^0_{\F_s}(\O;\Ls^{2}(\T^3))$, there exists a unique global ($2-$)solution $u$ to \eqref{eq:hyper_NS_s} satisfying 
\begin{align}
\label{eq:regularity_paths_local_high}
u\in L^2_{\loc}([s,\infty);H^1(\T^3;\R^3))\cap C([s,\infty);L^{2}(\T^3;\R^3))\  \text{ a.s.}
\end{align}
Moreover, if $\#\,\supp\theta<\infty$, then the global solution $u$ instantaneously regularizes in time and space, i.e.\ $u$ satisfies \eqref{eq:instantaneous_reg_local} with $\tau=\infty$ a.s.\
\end{theorem}

Although the previous result is not surprising, to the best of the author's knowledge, the global well-posedness of the stochastic HNSEs with high hyperviscosity was not present in the literature.

\smallskip

Finally, we discuss the global well-posedness of the HNSEs with transport noise with small data. 
The following will be used in the proof of Theorem \ref{t:main_theorem_L2}.

\begin{theorem}[Global well-posedness --  HNSEs with transport noise and small data]
\label{t:small_implies_global}
Let $1<\g<\frac{5}{4}$, $N_0\geq 1$ and $p=\frac{4\g}{6\g-5}$. Assume that 
$\theta\in \ell^2$ is radially symmetric and satisfies $\|\theta\|_{\ell^2}\leq N_0$. 
Then, for all $\varepsilon_0>0$, there exists $\delta_0=\delta_0(\g,N_0,\varepsilon_0)>0$ for which the following assertion holds:

For all $s\geq 0$ and $u_s\in L^p_{\F_s}(\O;\Bs^{5/2-2\g}_{2,p}(\T^3))$ with $\displaystyle{\Big|\int_{\T^3}u_s(x)\,\dd x\Big|\leq N_0}$ a.s.,
\begin{equation}
\label{eq:small_implies_global_statement}
\E\Big[\Big\|u_s-\int_{\T^3}u_s(x)\,\dd x\Big\|_{B^{5/2-2\g}_{2,p}(\T^3;\R^3)}^p\Big]\leq \delta_0
\ \quad \Longrightarrow \quad  
\P(\tau=\infty)>1-\varepsilon_0
\end{equation} 
where $\tau$ is the explosion time of the unique local $p$-solution $(u,\tau)$ to \eqref{eq:hyper_NS_s} provided by Theorem \ref{t:local_HNSEs_hyper_small_viscosity}. 
\end{theorem}

The proof of the above result is given in Subsection \ref{ss:global_small_hyper} below.
Although not clear at this stage, the fact that $\delta_0$ depends on $\theta$ only through its $\ell^2$-norm is the key point in using Theorem \ref{t:small_implies_global} to prove \eqref{eq:global_high_probability} in Theorem \ref{t:main_theorem_L2}. Finally, let us point out that the proof of the above result also yields the following stronger version of \eqref{eq:small_implies_global_statement}: If $\displaystyle{
\E\Big[\Big\|u_s-\int_{\T^3}u_s(x)\,\dd x\Big\|_{B^{5/2-2\g}_{2,p}}^p\Big]\leq \delta_0}$, then there exists a stopping time $\tau_0\leq \tau$ such that $\P(\tau_0=\infty)>1-\varepsilon_0$ and 
\begin{equation}
\label{eq:estimate_small_implies_global}
\E\Big[\sup_{t\in [0,\tau_0)}\Big\|u(t,\cdot)-\int_{\T^3}u_s(x)\,\dd x\Big\|_{B^{5/2-2\g}_{2,p}}^p\Big]\leq K_0
\E\Big[\Big\|u_s-\int_{\T^3}u_s(x)\,\dd x\Big\|_{B^{5/2-2\g}_{2,p}}^p \Big],
\end{equation} 
where $K_0$ depends only on $\g$ and $N_0$, and $p=\frac{4\g}{6\g-5}$ is as in Theorem \ref{t:small_implies_global}.
The proof of \eqref{eq:estimate_small_implies_global} is given at the end of Subsection \ref{ss:global_small_hyper}.

\subsection{Proof of Theorems \ref{t:local_HNSEs_hyper_small_viscosity} and \ref{t:local_HNSEs_hyper_high_viscosity}}
\label{ss:proofs_local}
The proofs of Theorems \ref{t:local_HNSEs_hyper_small_viscosity} and \ref{t:local_HNSEs_hyper_high_viscosity} are rather standard, and we content ourselves with giving a sketch. To prove Theorem \ref{t:local_HNSEs_hyper_small_viscosity}, we employ the theory of stochastic evolution equations in critical spaces developed in \cite{AV19_QSEE1,AV19_QSEE2}. To begin, let us reformulate the noise by using only real-valued noise coefficients and the independent real-valued Brownian motions $(B_{k,\alpha})_{k,\alpha}$ in \eqref{eq:complex_BM}. One can readily check that, for each real-valued vector field $u\in H^{\g}(\T^3;\R^3)$, 
$$
\sum_{k,\alpha} \p[(\sigma_{k,\alpha}\cdot\nabla) u]\, \dot{W}^{k,\alpha}
= \sum_{k,\alpha} \p [(\xi_{k,\alpha}\cdot \nabla) u]\, \dot{B}^{k,\alpha}
$$
where, for $x\in \T^3$, $\alpha\in \{1,2\}$ and $k\in \Z^3_0$ (see Subsection \ref{ss:noise}), 
$$
\xi_{k,\alpha}(x)\stackrel{{\rm def}}{=}
\left\{
\begin{aligned}
2\Re \sigma_{k,\alpha}(x)=2a_{k,\alpha}\cos(2\pi\,k\cdot x)\quad \text{ if } k\in \Z^3_+,\\
2\Im \sigma_{k,\alpha}(x)=2a_{k,\alpha}\sin(2\pi\,k\cdot x)\quad \text{ if } k\in \Z^3_-.
\end{aligned}
\right.
$$
Correspondingly, for each real-valued vector field $u\in H^\g(\T^3;\R^3)$,  
$$
\LL u= \frac{3\mu}{4}
\sum_{k,\alpha}\theta_k^2 \,\p \big[ \nabla \cdot ( \p[(\xi_{k,\alpha}\cdot \nabla)u]\otimes \xi_{k,\alpha}\big)\big].
$$
Hence, \eqref{eq:hyper_NS_s} can be rewritten as a stochastic evolution equation on the ground and state spaces $X_0=\Hs^{-\g}(\T^3)$ of the form:
\begin{equation}
\label{eq:SEE}
\left\{
\begin{aligned}
&\dd u+ A u\,\dd t =F(u)\,\dd t+ B u \,\dd \mathcal{B}_{\ell^2}, \\
&u(0)=u_0;
\end{aligned}
\right.
\end{equation}
where, for $u$ belonging to the state space $X_1= \Hs^{\g}(\T^3)$,
\begin{equation}
\begin{aligned}
\label{eq:choice_ABF}
&Au= (-\Delta)^{\g}- \LL u , \quad\quad F(u)=-\p [\nabla \cdot(u\otimes u)], \\
&\quad \quad Bu=\big(\textstyle{\sqrt{\frac{3\mu}{2}}}\,\theta_k\p[(\xi_{k,\alpha}\cdot\nabla)u]_{k,\alpha}\big)_{k,\alpha},
\end{aligned}
\end{equation}
and $\mathcal{B}_{\ell^2}$ denotes the $\ell^2$-cylindrical Brownian motion induced by $(B_{k,\alpha})_{k,\alpha}$, cf.\ \eqref{eq:cylindrical}.

To apply the results in \cite{AV19_QSEE1,AV19_QSEE2} (see also \cite{AV25_survey}), we employ the following estimates for the convective nonlinearity in the weak analytic setting for \eqref{eq:hyper_NS_s}.

\begin{lemma}[Nonlinear estimate]
\label{lem:nonlinearity_estimate}
For all $\g>1$, there exists $C_\g>0$ such that, for all $u,u'\in H^{\g}(\T^3;\R^3)$,
\begin{equation*}
\|\nabla \cdot (u\otimes u')\|_{H^{-\g}(\T^3;\R^3)}
\leq C_\g\|u\|_{H^{\beta}(\T^3;\R^3)}\|u'\|_{H^{\beta}(\T^3;\R^3)}
\end{equation*}
where $\beta=\frac{5}{4}-\frac{\g}{2}$ if $\g<\frac{5}{2}$, and $\beta>0$ arbitrary otherwise.
\end{lemma}

\begin{proof}
Let us first consider the case $\g<\frac{5}{2}$.
The Sobolev embedding $H^{-1,q}(\T^3)\embed H^{-\g}(\T^3)$ with $q=\frac{6}{1+2\g}$ implies (note that the latter holds as $q>1$)
\begin{align*}
\|\nabla\cdot (u\otimes u')\|_{H^{-\g}(\T^3;\R^3)}
\lesssim\|\nabla\cdot (u\otimes u')\|_{H^{-1,q}(\T^3;\R^3)}
\lesssim\|u\|_{L^{2q}(\T^3;\R^3)}\|u'\|_{L^{2q}(\T^3;\R^3)}
\end{align*}
for all $u,u'\in H^{\g}(\T^3;\R^3)$. In the case $\g<\frac{5}{2}$, the claimed estimate follows from the Sobolev embedding $H^{\beta}(\T^3)\embed L^{2q}(\T^3)$ as $\beta-\frac{3}{2}=-\frac{3}{2q}$.

If $\g\geq \frac{5}{2}$, then $H^{-1,q}(\T^3)\embed H^{-\g}(\T^d)$ for all $q\in (1,\infty)$, and reasoning as above the claim follows as $H^{\varepsilon}(\T^3)\embed L^{2q}(\T^3)$ where $\varepsilon = \frac{3}{2}(1-\frac{1}{q})$. The latter and the arbitrariness of $q\in (1,\infty)$ yields the claimed estimate also in the case $\g\geq \frac{5}{2}$.
\end{proof}

\begin{proof}[Proof of Theorem \ref{t:local_HNSEs_hyper_small_viscosity} -- Sketch]
From the above construction, it is clear that unique maximal $p$-solutions to \eqref{eq:hyper_NS} are equivalent to maximal $L^p_0$-local solutions in \cite[Definition 4.4]{AV19_QSEE1} due to \cite[Remark 4.6]{AV19_QSEE2} and Theorem \ref{t:smr_theta}\eqref{it:smr_theta_1}. 
Note that the nonlinear estimate of Lemma \ref{lem:nonlinearity_estimate} implies
$$
\|F(u)-F(u')\|_{H^{-\g}}\lesssim (1+\|u\|_{H^{\beta}}+ \|u'\|_{H^\beta})\|u-u'\|_{H^\beta}
$$
where $\beta=\frac{5}{4}-\frac{\g}{2}$. Since $[H^{-\g},H^{\g}]_{\varphi}=H^{\beta}$ with $\varphi= \frac{5}{8\g}+\frac{1}{4}$, the critical condition \cite[eq.\ (4.2)]{AV19_QSEE1} or \cite[eq.\ (4.3)]{AV25_survey} is equivalent to 
$$
2\varphi\leq 2 -\tfrac{1}{p} \quad \Longleftrightarrow \quad p\geq \tfrac{4\g}{6\g-5}
$$
as assumed in \eqref{eq:HNSEs_critical}. We emphasize that to ensure the condition $\varphi>1-\frac{1}{p}$ used in \cite{AV19_QSEE1,AV25_survey}, if necessary, one can replace the above parameter $\varphi$ by $\varphi \vee (1-\frac{1}{p}+\varepsilon)$ with $\varepsilon>0$ small.
The existence of unique maximal $p$-solution to \eqref{eq:hyper_NS} now follows from Theorem \ref{t:smr_theta}\eqref{it:smr_theta_1} and either \cite[Theorem 4.8]{AV19_QSEE1} or \cite[Theorem 4.7]{AV25_survey}. The instantaneous regularization follows almost verbatim from the corresponding statements in  \cite[Theorem 2.4 and 2.7]{AV21_NS} whose proofs rely on the results in \cite[Section 6]{AV19_QSEE2} (see also \cite[Subsection 5.3]{AV25_survey} for related results).
\end{proof}

\begin{proof}[Proof of Theorem \ref{t:local_HNSEs_hyper_high_viscosity} -- Sketch]
The existence of global unique solutions with paths in \eqref{eq:regularity_paths_local_high} can be obtained via Lemma \ref{lem:nonlinearity_estimate} and the `critical' variational setting \cite[Theorem 3.4]{AV24_variational}, where the assumption \cite[eq.\ (3.8)]{AV24_variational} follows from the standard cancellation 
$$\textstyle\int_{\T^3}\p [\nabla \cdot(v\otimes v)]\cdot v \,\dd x =0
$$ for   
$v\in \Hs^\g(\T^3)$.
The instantaneous regularization result assertion of Theorem \ref{t:local_HNSEs_hyper_high_viscosity} follows similarly to the one of \eqref{eq:instantaneous_reg_local} by following almost verbatim the proofs of the corresponding results for the NSEs in \cite[Theorems 2.4 and 2.7]{AV21_NS}. 
\end{proof}

\subsection{Proof of Theorem \ref{t:small_implies_global}}
\label{ss:global_small_hyper}
Here we prove Theorem \ref{t:small_implies_global}. Let us begin by collecting some useful facts. For notational convenience, in the proof of Theorem \ref{t:small_implies_global}, we let $s=0$. In the general case, the independence of the constant on $s\geq 0$ will be clear from the proof below. 
Note that the stochastic HNSEs \eqref{eq:hyper_NS_s} preserves the mean of the initial data, i.e, $\overline{u(t,\cdot)}=\overline{u_0}$ for all $t\in [0,\tau)$ a.s., where $\overline{\cdot}\stackrel{{\rm def}}{=}\int_{\T^3}\cdot\,\dd x$. 

Note that the mean-zero process $v=u-\overline{u_0}$ that is a unique $p$-solution to (as above, Definition \ref{def:p_solution} readily extend to the system below)
\begin{equation}
\label{eq:hyper_NS_s_v}
\left\{
\begin{aligned}
\partial_t v +  (\overline{u_0}\cdot \nabla) v
&+ \p[\nabla \cdot (v\otimes v)] 
+ (-\Delta)^{\g} v  \\
& = \sum_{k\in \Z^3_0}\sum_{\alpha\in \{1,2\}} \theta_k\, \p\big[(\sigma_{k,\alpha}\cdot\nabla) v \big]\circ \dot{W}^{k,\alpha}& \text{ on }&\T^3,\\
\nabla\cdot v&=0&\text{ on }&\T^3,\\
v(0,\cdot)&=v_0\stackrel{{\rm def}}{=} u_0 -\overline{u_0}&\text{ on }&\T^3.
\end{aligned}
\right.
\end{equation}
By a stopped version of the It\^o's formula for the functional $v\mapsto \|v\|_{L^2}^2$ (see e.g., \cite[Theorem 4.2.5]{LR15}) and $\nabla \cdot\sigma_{k,\alpha}=0$, the energy equality holds: 
\begin{equation}
\label{eq:energy_balance_v_revision}
\frac{1}{2}\,\|v(t)\|_{L^2}^2 + \int_0^t  \|(-\Delta)^{\g/2} v\|_{L^2}^2\,\dd r =\frac{1}{2}\, \|v_0\|_{L^2}^2 \ \text{ a.s.\ for all }t> 0.
\end{equation}
Let us emphasize that the above is obtained via the It\^o's formulation of \eqref{eq:hyper_NS_s_v} (cf.\ Definition \ref{def:p_solution}), and exploting that It\^o-Stratonovich corrector satisfies 
$$
\langle v,\LL v \rangle = -\frac{3\mu}{2}\sum_{k,\alpha}\theta_k^2 \int_{\T^3} \big|\p[(\sigma_{k,\alpha}\cdot \nabla) v ]\big|^2\,\dd x\ \ \text{ for all } v\in \Hs^{\g},
$$ 
where $\langle \cdot ,\cdot \rangle$ denotes the pairing in the duality $\Hs^{-\g}=(\Hs^\g)^*$. 
In particular, the latter contribution compensates the It\^o correction due to the transport noise.

From \eqref{eq:energy_balance_v_revision}, we deduce that $t\mapsto\|v(t)\|_{L^2}^2$ is absolutely continuous a.s. Therefore, as $\overline{v}=0$, by the Gr\"onwall and Poincar\'e inequalities it follows that
\begin{equation}
\label{eq:exponential_decay_L2_stochastic}
\|v(t)\|_{L^2}^2\leq e^{-c_0t} \|v_0\|_{L^2}^2 \ \text{ a.s.\ for all }t\geq 0,
\end{equation}
for some $c_0>0$ depending only on 
$\g$. We are now ready to prove Theorem \ref{t:small_implies_global}. 

\begin{proof}[Proof of Theorem \ref{t:small_implies_global}]
We split the proofs into three steps. Recall that  $p= \frac{4\g}{6\g-5}$, $N\geq 1$ is fixed and $\tau$ is the explosion time of the local $p$-solution to \eqref{eq:hyper_NS_s}, cf.\ Theorem \ref{t:local_HNSEs_hyper_small_viscosity}. Moreover, as in Lemma \ref{lem:nonlinearity_estimate}, we let $\beta=\frac{5}{4}-\frac{\g}{2}$.

\smallskip

\emph{Step 1: There exists $R>0$ depending only on $\g$ and $N_0\geq 1$ such that, for all $u_0\in L^p_{\F_0}(\O;B^{5/2-2\g}_{2,p})$ satisfying $|\overline{u_0}|\leq N_0$ a.s.\ and all stopping times $0\leq \lambda\leq \infty$ satisfying $\|v\|_{L^{2p}(0,\lambda;H^{\beta})}\leq M$ a.s.\ for some deterministic $M\geq 1$, we have}
\begin{equation}
\label{eq:estimate_psiR}
\E \big[\psi_R(\|v\|_{L^{2p}(0,\lambda;H^{\beta})}^p)\big]
\leq \E\|v_0\|_{B^{5/2-2\g}_{2,p}}^p,
\end{equation}
\emph{where $\psi_R(x)=R^{-1}x-x^2$ for $x\geq 0$.}

By standard interpolation theory we have $\|u\|_{H^{\beta}}\lesssim \|u\|_{B^{5/2-2\g}_{2,p}}^{1/2}\|u\|_{H^{\g}}^{1/2}$ and therefore
\begin{equation}
\label{eq:embedding_paths_smr_global}
L^{\infty}(0,t;B^{5/2-2\g}_{2,p})\cap L^p(0,t;H^\g)\embed
L^{2p}(0,t;H^{\beta}) \ \text{ for all }t>0,
\end{equation}
and the constant in the embedding can be made uniform in $t>0$.

Combining the latter, Theorem \ref{t:smr_theta}\eqref{it:smr_theta_2} with $w=-\overline{u_0}$ and Lemma \ref{lem:nonlinearity_estimate}, 
\begin{align*}
\E [\|v\|_{L^{2p}(0,\lambda;H^{\beta})}^p]
&\leq R_0\big( \E[\|v_0\|_{B^{5/2-2\g}_{2,p}}^p] + \E[ \|v\|_{L^{2p}(0,\lambda;H^{\beta})}^{2p}]
+ \E[ \|v\|_{L^{p}(0,\lambda;L^2)}^p ]\big) \\
&\stackrel{\eqref{eq:exponential_decay_L2_stochastic}}{\leq} 
R_0\big( R_1\E[\|v_0\|_{B^{5/2-2\g}_{2,p}}^p] + \E [\|v\|_{L^{2p}(0,\lambda;H^{\beta})}^{2p}] \big) 
\end{align*}
where $R_0$ and $R_1$ depend only on $\g$ and $N_0$.
This proves the claim of Step 1. 

\smallskip

\emph{Step 2: For all $\varepsilon_0\in (0,1)$, we have}
$$
\E[\|v_0\|_{B^{5/2-2\g}_{2,p}}^p]\leq \frac{\varepsilon_0}{8R^2}\qquad \Longrightarrow \qquad
\P\Big(\|v\|_{L^{2p}(0,\tau;H^{\beta})}\leq \frac{1}{(2R)^{1/p}}\Big)>1-\varepsilon_0.
$$

Let us assume that $\E[\|v_0\|_{B^{5/2-2\g}_{2,p}}^p]\leq \frac{\varepsilon_0}{8R^2}$. By contradiction, suppose that 
\begin{equation}
\label{eq:contradiction_small_global}
\P(\Dom)\leq 1-\varepsilon_0 \quad \text{ where }\quad
\Dom\stackrel{{\rm def}}{=} \Big\{\|v\|_{L^{2p}(0,\tau;H^{\beta})}\leq \frac{1}{(2R)^{1/p}}\Big\}.
\end{equation}
Let $\psi_R$ be as in Step 1. Clearly, $\psi_R$ has a global maximum given by $1/(4R^2)$, and it is attained at $1/(2R)$.
Note that $u\in L^{2p}_{\loc}([0,\tau);H^\beta)$ a.s.\ by \eqref{eq:regularity_paths_local_small_1}-\eqref{eq:regularity_paths_local_small_2} and \eqref{eq:embedding_paths_smr_global} and a.s.\ the mapping $[0,\tau)\ni t \mapsto \|v\|_{L^{2p}(0,t;H^{\beta})}=\|u-\overline{u_0}\|_{L^{2p}(0,t;H^{\beta})}$ is continuous a.s.\ and starts at $0$.
Thus, for a.a.\ $\om\in\O\setminus \Dom$ there exists $t<\tau(\om)$ such that $\|v(\cdot,\om)\|_{L^{2p}(0,t;H^{\beta})}> (2R)^{-1/p}$ (see Figure \ref{fig:1} for a schematic picture of the argument below). Therefore, $
\lambda_R<\tau$ a.s.\ on $\O\setminus\Dom$, where $\lambda_R$ is the stopping time given by 
\begin{equation}
\label{eq:def_LambdaR}
\lambda_R\stackrel{{\rm def}}{=}\inf\Big\{t\in [0,\tau)\,:\, \|v\|_{L^{2p}(0,t;H^{\beta})}\geq \frac{1}{(2R)^{1/p}}\Big\}
\end{equation}
where $\inf\emptyset \stackrel{{\rm def}}{=} \tau$. From the above observation, it follows that 
\begin{align}
\label{eq:property_psi_lambda_1}
\psi_R (\|v\|_{L^{2p}(0,\lambda_R;H^{\beta})}^p)&\geq 0\text{ a.s.\ on }\Dom,\\
\label{eq:property_psi_lambda_2}
\psi_R (\|v\|_{L^{2p}(0,\lambda_R;H^{\beta})}^p)&= \frac{1}{4R^2}\text{ a.s.\ on }\O\setminus\Dom.
\end{align}
Now, 
if \eqref{eq:contradiction_small_global} holds, then
applying the estimate of Step 1 with $\lambda=\lambda_R$, 
\begin{align*}
\E \big[\psi_R(\|v\|_{L^{2p}(0,\lambda_R;H^{\beta})}^p)\big]
&\ \stackrel{\eqref{eq:property_psi_lambda_1}}{\geq} 
\E \big[\one_{\O\setminus \Dom}\,\psi_R(\|v\|_{L^{2p}(0,\lambda_R;H^{\beta})}^p)\big]\\
&\ \stackrel{\eqref{eq:property_psi_lambda_2}}{=}\frac{1}{4R^2} \P(\O\setminus\Dom)\\
&\ \stackrel{\eqref{eq:contradiction_small_global}}{\geq}  \frac{\varepsilon_0}{4R^2}>\E[\|v_0\|_{B^{5/2-2\g}_{2,p}}^p],  
\end{align*}
where in the last step we used that $\E[\|v_0\|^p_{B^{5/2-2\g}_{2,p}}]\leq \varepsilon_0/(8R^2)$ and $\varepsilon_0>0$.
The above contradicts Step 1. Therefore \eqref{eq:contradiction_small_global} is false in case $\E[\|v_0\|^p_{B^{5/2-2\g}_{2,p}}]\leq \varepsilon_0/(8R^2)$. Hence, the claim of Step 2 follows.

\begin{figure}[h!]
\begin{center}
\begin{tikzpicture}[scale=0.7]
\draw[-stealth, line width=0.25mm] (0, 0) -- (6.5, 0);
\draw[-stealth, line width=0.25mm] (0, 0) -- (0, 4.5);
\draw[scale=0.7, domain=0:7, smooth, variable=\x, Red, line width=0.5mm] plot ({\x}, {\x*3-0.05*\x*\x*9});
\draw[scale=0.7, domain=0:1.4, smooth, variable=\x, black, line width=0.5mm] plot ({\x}, {\x*3-0.05*\x*\x*9}) node[right, xshift=-0.75cm, yshift=0.25cm] {$\x(t)$};
\draw (0,3) -- (6,3) node[above, xshift=1.1cm, yshift=-0.5cm]{$\quad \E\|u_0\|_{B^{\g(1-2/p)}_{2,p}}^p$};
\draw[dashed] (0,3.52) -- (6,3.52) node[left, xshift=-4.2cm]{$\displaystyle{\frac{1}{4R^2}}$}  ;
\draw (0,0) -- (2.9,4.388) node[left, xshift=0.7cm, yshift=-0.05cm]{$\displaystyle{\frac{x}{2R}}$}  ;

\filldraw[black] (2.32,3.51) circle (2pt); 
\node[Red] at (5.2,1.5) {$\psi_{R}(x)$};
\draw[dashed] (2.32,-0.1) node[left, xshift=0.6cm, yshift=-0.5cm]{$\displaystyle{1/(2R)}$} -- (2.32,3.52);
\end{tikzpicture}
\caption{Illustration for Step 2 of Theorem \ref{t:small_implies_global} -- $\x(t)=
\|v\|_{L^{2p}(0,t;H^{\beta_0})}^p$.}
\label{fig:1}
\end{center}
\end{figure}

\emph{Step 3: Conclusion}. Let $\delta_0=\varepsilon_0/(8R^2)$ and assume 
$\E[\|v_0\|_{B^{5/2-2\g}_{2,p}}^p]\leq \delta_0$.
Again, by Lemma \ref{lem:nonlinearity_estimate} and \cite[Theorem 4.10(1)]{AV19_QSEE2}, it follows that the $p$-solution $(u,\tau)$ satisfies the following blow-up criterion:
\begin{equation}
\label{eq:blow_up_criterion_small_implies_global}
\P(\tau<\infty, \|u\|_{L^{2p}(0,\tau;H^{\beta})}<\infty)=0.
\end{equation}
Recall that, by Step 2, $\P(\Dom)>1-\varepsilon_0$ where $\Dom\stackrel{{\rm def}}{=}\{\|v\|_{L^{2p}(0,\tau;H^{\beta})}\leq r_0\}$ with $r_0=1/(2R)^{1/p}$. Hence,
\begin{align*}
\P(\{\tau<\infty\}\cap \Dom)
&= \P(\{\tau<\infty, \|v\|_{L^{2p}(0,\tau;H^{\beta})}<\infty\}\cap \Dom)\\
&\stackrel{(i)}{=} \P(\{\tau<\infty, \|u\|_{L^{2p}(0,\tau;H^{\beta})}<\infty\}\cap \Dom)
\stackrel{\eqref{eq:blow_up_criterion_small_implies_global}}{=} 0,
\end{align*}
where in $(i)$ we used that $\|v\|_{L^{2p}(0,t;H^{\beta})}<\infty$ implies $ \|u\|_{L^{2p}(0,t;H^{\beta})}<\infty$ if $t<\infty$.
In particular, $\tau=\infty$ a.s.\ on $\Dom$. This concludes the proof of Theorem \ref{t:small_implies_global}.
\end{proof}

\begin{proof}[Proof of the estimate \eqref{eq:estimate_small_implies_global}]
Let $R$ and $\lambda_R$ be as in the proof of Theorem \ref{t:small_implies_global}. Steps 2 and 3 yield that $\lambda_R=\tau=\infty$ a.s.\ on $\Dom$, where the latter is as in \eqref{eq:contradiction_small_global}.
Since
$
\frac{x}{2R}\leq \psi_{R}(x)$ for all $ x\in \big[0, \frac{1}{2R}\big]$ (see Figure \ref{fig:1}),
from \eqref{eq:estimate_psiR} it follows that 
$$
\E \big[\|v\|_{L^{2p}(0,\lambda_R;H^{\beta})}^p\big]
\leq 2R \,\E[\|v_0\|_{B^{5/2-2\g}_{2,p}}^p],
$$
where $v=u-\int_{\T^3}u_0(x)\,\dd x$. 
Now the estimate \eqref{eq:estimate_small_implies_global} is a consequence of Theorem \ref{t:smr_theta}\eqref{it:smr_theta_2} and the definition of $\lambda_R$, see \eqref{eq:def_LambdaR}. 
\end{proof}

\section{Global well-posedness of HNSEs with high viscosity}
\label{s:global_high_viscosity}
In this section, we establish the global well-posedness of HSNEs on $\T^3$ with high viscosity $\mu\gg1$:
\begin{equation}
\label{eq:hyper_NS_det}
\left\{
\begin{aligned}
\partial_t u &+ \p[\nabla \cdot(u\otimes u)] 
+ (-\Delta)^{\g} u -\mu\Delta u =0 & \text{ on }&\T^3,\\
\nabla\cdot u&=0&\text{ on }&\T^3,\\
u(0,\cdot)&=u_0&\text{ on }&\T^3.
\end{aligned}
\right.
\end{equation}

More precisely, we prove the following result.

\begin{theorem}[Global well-posedness of HNSEs with high viscosity]
\label{t:global_high_viscosity}
Let $1<\g<\frac{5}{4}$. Fix $N\geq 1$ and $\frac{5}{2}-2\g<\delta<\g$.
Then there exists $\mu_0>0$ depending only on $\delta,\g$ and $N$ such that, if $\mu\geq \mu_0$, then for all initial data $u_0$ satisfying
\begin{equation}
\label{eq:initial_data_bounded_N}
u_0\in \Hs^{\delta}(\T^3)\quad \text{ and }\quad \|u_0\|_{H^{\delta}(\T^3;\R^3)}\leq N,
\end{equation}
the {\normalfont{HSNEs}} \eqref{eq:hyper_NS_det} has a unique global solution 
\begin{equation}
\label{eq:regularity_deterministic_solution}
u\in H^1_{\loc}([0,\infty);\Hs^{\delta-\g}(\T^3))\cap L^2_{\loc}([0,\infty);\Hs^{\delta+\g}(\T^3))\subseteq C([0,\infty);\Hs^{\delta}(\T^3)),
\end{equation} 
satisfying
\begin{equation}
\label{eq:estimate_boundedness_high_viscosity}
\textstyle
\big\|u-\int_{\T^3}u_0(x)\,\dd x \big\|_{L^2(\R_+;H^{\delta+\g}(\T^3;\R^3))}
+
\sup_{t>0}\|u(t)\|_{H^{\delta}(\T^3;\R^3)}\lesssim_{\delta,\g,N} \|u_0\|_{H^\delta(\T^3;\R^3)}.
\end{equation}
Moreover, for all $\delta_0<\delta$, there exists $\eta_0>0$ depending only on $\delta,\g,N,\delta_0$ such that
\begin{equation}
\label{eq:estimate_boundedness_high_viscosity_decay}
\textstyle
\big\|u(t)-\int_{\T^3}u_0(x)\,\dd x \big\|_{H^{\delta_0}(\T^3;\R^3)}\leq_{\delta,\delta_0,\g,N}  e^{-\eta_0 t} \|u_0\|_{H^\delta(\T^3;\R^3)}\ \  \text{ for all }t>0.
\end{equation}
\end{theorem}

The result above heavily relies on the fact that $H^{\delta}$ with $\delta>\frac{5}{2}-2\g$ is \emph{subcritical} for the HNSEs \eqref{eq:hyper_NS_det}, as discussed in Subsection \ref{ss:novelty}. We expect this result to no longer hold in the case of critical data. Moreover, the reader may verify that the argument presented below remains valid even if the lower-order dissipation term $-\mu\Delta u$ is replaced by the (seemingly) weaker dissipation term $-\mu u$. 

The proof of Theorem \ref{t:global_high_viscosity} follows a similar approach to that of Theorem \ref{t:small_implies_global}. However, in this case, $L^p$-theory is not required, and the proof can be entirely framed within a Hilbert space setting. As commented in Subsection \ref{ss:novelty}, the latter is \emph{not} true for Theorem \ref{t:small_implies_global}. 
For later use, let us note that, if $u_0\in \Bs^{\g(1-2/p)}_{2,p}(\T^3)$ for some $p>\frac{4\g}{6\g-5}$, then the solution $u$ provided by Theorem \ref{t:global_high_viscosity} also satisfy 
\begin{equation}
\label{eq:comparison_p_solution_etc}
u\in L^p_{\loc}([0,\infty);\Hs^\g(\T^3))\cap C([0,\infty);\Bs^{\g(1-2/p)}_{2,p}(\T^3)).
\end{equation}
The latter can be shown by exploiting the subcriticality of the condition $p>\frac{4\g}{6\g-5}$ and the uniqueness of $p$-solutions as in Theorem \ref{t:local_HNSEs_hyper_small_viscosity}. In particular, in this case, $u$ is a $p$-solution to \eqref{eq:hyper_NS_det} (the latter can be defined as in Definition \ref{def:p_solution}).

\begin{proof}
For clarity, we divide the proof into several steps. 

\smallskip

\emph{Step 1: For all $v,v'\in \Hs^\delta$ and $\delta\in [0,\g]$, we have}
\begin{equation}
\label{eq:interpolation_deterministic}
\|\nabla\cdot(v\otimes v')\|_{H^{\delta-\g}}\lesssim \|v\|_{H^{\beta}}\|v'\|_{H^{\beta}}
\end{equation}
\emph{where $\beta=\frac{5}{4}-\frac{\g}{2}+\frac{\delta}{2}$.}
By Lemma \ref{lem:nonlinearity_estimate} and bilinear interpolation (see e.g., \cite[Theorem 4.1.1]{BeLo}), it is enough to show \eqref{eq:interpolation_deterministic} in the case $\delta=\g$.
The latter case follows by noticing that $\nabla\cdot(v\otimes v')=(v\cdot\nabla)v'$ as $\nabla \cdot v=0$ and
$$
\|(v\cdot\nabla)v'\|_{L^2}\lesssim\|v\|_{L^{12}}\|\nabla v'\|_{L^{(12)/5}}\lesssim\|v\|_{H^{5/4}}\|v'\|_{H^{5/4}}, 
$$
where we used the H\"older inequality and Sobolev embeddings.

\smallskip

For the remainder of the proof, we fix $\delta$ such that  
\begin{equation*}
 \tfrac{5}{2}-2\g<\delta <\g.
\end{equation*}
From \eqref{eq:interpolation_deterministic} and standard well-posedness results for parabolic PDEs yield the existence of a unique local solution $(u,\tau)$ satisfying
\begin{equation}
\label{eq:regularity_u_det}
u\in H^1_{\loc}([0,\tau);\Hs^{\delta-\g}(\T^3))\cap  L^2_{\loc}([0,\tau);\Hs^{\delta+\g}(\T^3))\subseteq C([0,\tau);\Hs^\delta(\T^3)).
\end{equation}
Moreover, we also have
\begin{equation}
\label{eq:blow_up_deterministic}
\textstyle
\tau<\infty \qquad \Longrightarrow \qquad \sup_{t\in [0,\tau) }\|u(t)\|_{H^{\delta}}=\infty.
\end{equation}
The claims \eqref{eq:regularity_u_det}-\eqref{eq:blow_up_deterministic} follow, e.g., from either \cite[Theorem 2.1 and Corollary 2.3(i)]{CriticalQuasilinear} or as in the proof of Theorem \ref{t:local_HNSEs_hyper_small_viscosity}.

As in the proof of Theorem \ref{t:small_implies_global}, it is convenient to reformulate \eqref{eq:hyper_NS_det} in terms of $v=u-\int_{\T^3} u_0(x)\,\dd x$, where $\overline{u_0}=\int_{\T^3} u_0(x)\,\dd x$. Note that $\int_{\T^3} u(t,x)\,\dd x=\int_{\T^3} u_0(x)\,\dd x$ as the deterministic HNSEs \eqref{eq:hyper_NS_det} preserve the mean. Thus, $(v,\tau)$ is a local solution to 
\begin{equation}
\label{eq:hyper_NS_det_v}
\left\{
\begin{aligned}
\partial_t v&+(\overline{u_0}\cdot \nabla) v+ \p[\nabla \cdot(v\otimes v)] 
+ (-\Delta)^{\g} v -\mu\Delta v =0 & \text{ on }&\T^3,\\
\nabla\cdot v&=0&\text{ on }&\T^3,\\
v(0,\cdot)&=v_0\stackrel{{\rm def}}{=}u_0-\overline{u_0}&\text{ on }&\T^3.
\end{aligned}
\right.
\end{equation}
Next, we discuss two consequences of energy inequality. By testing \eqref{eq:hyper_NS_det_v} with $v$ and using the usual cancellation $\int_{\T^3} (\overline{u_0}\cdot\nabla) v\cdot v\,\dd x =\int_{\T^3} [\nabla\cdot (v\otimes  v)] \cdot v\,\dd x =0$ for $v\in \Hs^\g$, the solution $v$ to \eqref{eq:hyper_NS_det_v} satisfies, for all $t<\tau$,
\begin{equation}
\label{eq:energy_inequality_mu_global}
\frac{1}{2}\,\|v(t)\|_{L^2}^2+ \int_0^t \|(-\Delta)^{\g/2} v\|_{L^2}^2\,\dd s 
 +\mu \int_0^t \|\nabla v(s)\|_{L^2}^2\,\dd s = \frac{1}{2}\,\|v_0\|_{L^2}^2.
\end{equation}
In particular, by the Poincar\'e inequality, for some $c_0$ depending only on $\g$,
\begin{align}
\label{eq:energy_inequality_mu}
\|v(t)\|_{L^2}^2\leq e^{-tc_0}\|u_0\|_{L^2}^2 \text{ for all $t\in (0,\tau)$.}
\end{align}
In the remaining part of the proof, we verify that  $\tau=\infty$ and that \eqref{eq:estimate_boundedness_high_viscosity} holds. The estimate in \eqref{eq:estimate_boundedness_high_viscosity_decay} follows by interpolating \eqref{eq:estimate_boundedness_high_viscosity} and \eqref{eq:energy_inequality_mu} with $\tau=\infty$. 
To prove \eqref{eq:estimate_boundedness_high_viscosity}, we use standard energy methods. More precisely, computing $\frac{1}{2}\frac{\dd}{\dd t}\|v(t)\|_{H^{\delta}}^2$, it follows from Step 1 that 
\begin{align*}
\textstyle
\frac{\dd}{\dd t}\|v(t)\|_{H^{\delta}}^2
+\|v(t)\|_{H^{\delta+\g}}^2&+\mu \|v(t)\|_{H^{\delta+1}}^2 \\
&\leq C_0 \|\p[\nabla \cdot(v\otimes v)]\|_{H^{\delta-\g}}^2 + C_0 \|(\overline{u_0}\cdot \nabla) v\|_{H^{\delta-\g}}^2\\
&\leq C_1 \|v\|^4_{H^\beta} + C_1 \| v\|_{H^{\delta-\g+1}}^2\\
&\leq C_1 \|v\|^4_{H^\beta} + C_2 \| v\|_{L^2}^2+ \tfrac{1}{2}\|v\|_{H^{\delta+\g}}^2
\end{align*}
for a.a.\ $t<\tau$, 
where $C_0,C_1$ and $C_2$ are constants depending only on $\delta,\g$ and $N$ and $\beta=\frac{5}{4}-\frac{\g}{2}+\frac{\delta}{2}$. In particular, by \eqref{eq:energy_inequality_mu}, for all $t<\tau$, 
\begin{equation}
\label{eq:energy_estimate_hyperviscous}
\textstyle
\sup_{s\in [0,t]}
\|v(s)\|_{H^{\delta}}
+\|v\|_{L^2(0,t;H^{\delta+\g})} \lesssim_{\delta,\g,N} \|u_0\|_{H^\delta} +
\|v\|_{L^4(0,t;H^{\beta})}^2.
\end{equation}
In the following step, we show how to estimate the last term on the RHS\eqref{eq:energy_estimate_hyperviscous}. 
\smallskip

\emph{Step 2: Let $\delta$ be as above. There exist $\beta_0>\beta$ and a constant $K_0>0$ such that, for all $t<\tau$,} 
\begin{equation}
\label{eq:interpolation_estimate_deterministic}
\|v\|_{L^{4}(0,t;H^{\beta_0})}\leq K_0 \|v\|_{L^\infty(0,t;H^{\delta})}^{1/2}\|v\|_{L^2(0,t;H^{\delta+\g})}^{1/2}.
\end{equation} 

The above should be compared with the one obtained in Step 1 in the proof of Theorem \ref{t:small_implies_global}.
Note that the subcriticality of the space $H^{\delta}$ is encoded in the condition $\beta_0>\beta$.

Arguing as in \eqref{eq:embedding_paths_smr_global}, by standard interpolation estimates, we have
$
L^{\infty}(0,t;H^{\delta})\cap L^2(0,t;H^{\delta+\g})\embed L^4(0,t;H^{\delta+\g/2}),
$
and with a corresponding estimate independent of $t>0$. Hence, \eqref{eq:interpolation_estimate_deterministic} follows from the latter and $\beta_0\stackrel{{\rm def}}{=}\delta+\g/2>\beta$ as $\delta>\frac{5}{2}-2\g$. 

\smallskip

\emph{Step 3: Let $\beta$ and $\beta_0$ be as in Step 1 and 2, respectively. Then there exists $R(\delta,\g,N)>0$ such that, for all $t<\tau$ and $\mu>0$,
\begin{equation}
\label{eq:estimate_R_high_viscosity_not_good_form}
\|v\|_{L^{4}(0,t;H^{\beta_0})}
\leq  R\|u_0\|_{H^\delta}+ R\mu^{-\coeff}\|v\|_{L^{4}(0,t;H^{\beta_0})}^2
\end{equation}
where $\coeff>0$ depends only on $\delta$ and $\g$.}

Let us begin with observing that \eqref{eq:energy_inequality_mu_global} and interpolation yield, for all $t<\tau$,
\begin{align*}
\|v\|_{L^{4}(0,t;L^2)}
&\lesssim\|v\|_{L^{4}(0,t;H^{1/2})}\\
&\lesssim \|v\|_{L^{2}(0,t;H^1)}^{1/2} \|v\|_{L^{\infty}(0,t;L^2)}^{1/2}\\
&= \mu^{-1/4} \big(\mu^{1/2}\|v\|_{L^{2}(0,t;H^1)}\big)^{1/2}\|v\|_{L^{\infty}(0,t;L^2)}^{1/2}\\
&\lesssim \mu^{-1/4} \|u_0\|_{H^{\delta}},
\end{align*} 
for all $t<\tau$. Note also that the implicit constant is independent of $t$.
Moreover, as $[L^2,H^{\beta_0}]_{r_0}=H^{\beta}$ where $r_0=\beta/\beta_0\in(0,1 )$,  \eqref{eq:energy_estimate_hyperviscous} and \eqref{eq:interpolation_estimate_deterministic} imply, for all $t<\tau$,
\begin{align*}
\|v\|_{L^4(0,t;H^{\beta_0})}
&\lesssim \|u_0\|_{H^\delta}
+\|v\|_{L^{4}(0,t;L^{2})}^{2(1-r_0)}
\|v\|_{L^{4}(0,t;H^{\beta_0})}^{2r_0}\\
&\lesssim \|u_0\|_{H^\delta}
+ \mu^{-(1-r_0)/2}\|u_0\|_{H^{\delta}}^{2(1-r_0)}
\|v\|_{L^{4}(0,t;H^{\beta_0})}^{2r_0}\\
&\lesssim \|u_0\|_{H^{\delta}}+
\|u_0\|_{H^{\delta}}^{2}+
\mu^{-(1-r_0)/(2r_0)}\|v\|_{L^{4}(0,t;H^{\beta_0})}^2;
\end{align*}
again, with an implicit constant independent of $t>0$. The previous and the assumption $\|u_0\|_{H^\delta}\leq N$ lead to \eqref{eq:estimate_R_high_viscosity_not_good_form}.

The estimate \eqref{eq:estimate_R_high_viscosity_not_good_form}
is the key point to prove Theorem \ref{t:global_high_viscosity}. Indeed, even if the RHS grows quadratically w.r.t.\ $\|v\|_{L^{2p}(0,t;H^{\beta_0})}$, we can tune the parameter $\mu>0$ so that the constant in front of such a quadratic term is as small as needed.
To exploit such an idea, let us rewrite \eqref{eq:estimate_R_high_viscosity_not_good_form} in the form
\begin{equation}
\label{eq:estimate_R_high_viscosity}
\psi_{R,\mu} (\|v\|_{L^{2p}(0,t;H^{\beta_0})}) 
\leq \|u_0\|_{H^\delta}  \ \ \text{for $t\in (0,\tau)$},
\end{equation} 
where $
\psi_{R,\mu} (x)= R^{-1}x-\mu^{-\coeff} x^2$ and $ x\in \R$. 

\smallskip

\emph{Step 4: If $\mu\geq  (8R^2N)^{1/\coeff}$, then}
$$
\|v\|_{L^{4}(0,\tau;H^{\beta_0})}\leq \mu^{\coeff}/(2R) \quad \text{ \emph{and} }\quad \|v\|_{L^{4}(0,\tau;H^{\beta_0})}\leq 2R \|u_0\|_{H^\delta}.
$$

We prove the claim by contradiction. 
The argument is now similar to the one used in Step 2 of Theorem \ref{t:small_implies_global}. For a schematic picture, the reader is referred to Figure \ref{fig:1}. Assume that $\mu\geq  (8R^2N)^{1/\coeff}$ and 
\begin{equation}
\label{eq:contradiction_Lbeta0}
\|v\|_{L^{4}(0,\tau;H^{\beta_0})}>\mu^{\coeff}/(2R).
\end{equation}
Note that $
\|v\|_{L^{4}(0,\tau;H^{\beta_0})}=\sup_{0\leq t\leq \tau}
\|v\|_{L^{4}(0,t;H^{\beta_0})}
$ and $[0,\tau)\ni t\mapsto \|v\|_{L^{2p}(0,t;H^{\beta_0})}$ is a continuous function starting at 0. Hence, \eqref{eq:contradiction_Lbeta0} implies the existence of a time $t_0<\tau$ for which 
\begin{equation}
\label{eq:t_0_matches_maximum}
\|v\|_{L^{4}(0,t_0;H^{\beta_0})}=\mu^{\coeff}/(2R).
\end{equation}
We now prove that \eqref{eq:t_0_matches_maximum} contradicts \eqref{eq:estimate_R_high_viscosity}. Indeed, if \eqref{eq:t_0_matches_maximum} holds, then 
\begin{align*}
\psi_{R,\mu}(\|v\|_{L^{4}(0,t_0;H^{\beta_0})})
= \psi_{R,\mu}\Big(\frac{\mu^{\coeff}}{2R}\Big) 
= \frac{\mu^{{\coeff}}}{4R^2}
\stackrel{(i)}{\geq } 2N >\|u_0\|_{H^\delta}
\end{align*}
where in $(i)$ we used the assumption $\mu\geq  (8R^2N)^{1/{\coeff}}$. The above contradicts \eqref{eq:estimate_R_high_viscosity} and therefore Step 1 is proved.

To prove the claimed estimate, it is enough to note that (see Figure \ref{fig:1} for a similar situation)
$
\frac{x}{2R}\leq \psi_{R,\mu}(x)$ for all $x\in \big[0, \frac{\mu^{\coeff}}{2R}\big].
$
Hence, the claim of Step 4 follows from \eqref{eq:estimate_R_high_viscosity} and the previously established fact that $
\|v\|_{L^{4}(0,\tau;H^{\beta_0})}\leq\frac{\mu^{\coeff}}{2R}$.

\smallskip

\emph{Step 5: Proof of $\tau=\infty$ and \eqref{eq:estimate_boundedness_high_viscosity}}. Now, assume that $\mu\geq  (8R^2N)^{1/{\coeff}}$. Hence,  
\begin{align*}
\textstyle
\sup_{t\in [0,\tau)}\|v(t)\|_{H^\delta}+
\|v\|_{L^2(0,\tau;H^{\delta+\g})}
&\stackrel{\eqref{eq:energy_estimate_hyperviscous}}{\leq} C_0\big(\|u_0\|_{H^\delta}+ \|v\|_{L^4(0,\tau;H^{\beta})}^2\big)\\
& \ \ \leq C_0\big(\|u_0\|_{H^\delta}+2R \|u_0\|_{H^\delta}^2\big),
\end{align*}
where $C_0>0$ is independent of $\mu$, and in the last step we used Step 4. Since $\|u_0\|_{H^\delta}\leq N$, it follows that $\overline{u_0}\leq N$ and
\begin{equation}
\textstyle
\label{eq:estimate_boundedness_high_viscosity_tau}
\|u-\overline{u_0}\|_{L^2(0,\tau;H^{\delta+\g})}+ \sup_{t\in [0,\tau)}\|u(t)\|_{H^\delta}\leq R_0\|u_0\|_{H^\delta}
\end{equation}
for some constant $R_0$ depending only on $\delta,\g$ and $N$.

From the above and the blow-up criterion in \eqref{eq:blow_up_deterministic}, it follows that $\tau=\infty$. 
Therefore, \eqref{eq:estimate_boundedness_high_viscosity} follows from \eqref{eq:estimate_boundedness_high_viscosity_tau}. 
As mentioned below \eqref{eq:energy_inequality_mu}, the latter with $\tau=\infty$ and \eqref{eq:estimate_boundedness_high_viscosity} yield \eqref{eq:estimate_boundedness_high_viscosity_decay}.
This concludes the proof of Theorem \ref{t:global_high_viscosity}.
\end{proof}

\section{Global smooth solutions of stochastic HNSEs via scaling limit}
\label{s:scaling_limit}
In this section, we finally give the proof of Theorem \ref{t:main_theorem_L2}. 
The key ingredient is a scaling limit for the truncated version of the HNSEs \eqref{eq:hyper_NS}:
\begin{equation}
\label{eq:hyper_NS_cut_off}
\left\{
\begin{aligned}
\partial_t v&+\phi_{R,r}(v)\,\p [\nabla \cdot (v\otimes v)] 
&\\
&=-(-\Delta)^{\g} v+\sqrt{\frac{3\mu}{2}}\sum_{k,\alpha} \theta_k\,\p[(\sigma_{k,\alpha}\cdot\nabla) v ]\circ \dot{W}^{k,\alpha} &\text{ on }&\T^3,\\
\vspace{0.1cm}
v(0,\cdot)&=v_0&\text{ on }&\T^3,
\end{aligned}
\right.
\end{equation}
where, for $R>0$, $r\in (0,\g)$ and a smooth function satisfying $\phi=1$ on $[0,1]$ and $\phi=0$ on $[2,\infty)$, we set
\begin{equation}
\label{eq:definition_cutoff_phiRr}
\phi_{R,r}(v)\stackrel{{\rm def}}{=}\phi\big(R^{-1} \|v\|_{H^r(\T^3;\R^3)}\big).
\end{equation}
Finally, as above, $\p$ is the Helmholtz projection defined in Subsection \ref{ss:helmholtz_related}. 

We perform the above mentioned scaling result for \eqref{eq:hyper_NS_cut_off} with the scale parameter $\theta=\theta^n$
 given by
\begin{equation}
\label{eq:choice_theta_n}
\theta^n =\frac{\Theta^n}{\|\Theta^n\|_{\ell^2}} \qquad \text{ with } \qquad \Theta^n_k =\frac{1}{|k|^\alpha}\one_{\{n \leq |k|\leq 2n\} } \ \text{ for }k\in \Z^3_0,
\end{equation}
where $\alpha>0$. One can check that the above sequence satisfies (cf.\ \cite[eq.\ (1.9)]{L23_enhanced})
\begin{equation}
\label{eq:property_sequence_thetan}
\|\theta^n\|_{\ell^2}=1 \ \text{ and }\ \#\, \supp\theta^n<\infty\ \text{ for all } n\geq 1, \qquad \lim_{n\to \infty}\|\theta^n\|_{\ell^\infty}=0.
\end{equation}
The reason to consider the sequence \eqref{eq:choice_theta_n} comes from the result \cite[Theorem 5.1]{FL19} (see also \cite[Theorem 3.1]{L23_enhanced}) where an explicit convergence of the It\^o-Statonovich correction $\LLnn$ defined in \eqref{eq:Ito_stratonovich_change} as $n\to \infty$ is proved.

\smallskip

To state the main result of this section, let us note that unique, local, and maximal $p$-solutions to \eqref{eq:hyper_NS_cut_off} can be defined analogously to those of \eqref{eq:hyper_NS} given in Definition \ref{def:p_solution}. In addition, we say that $(v,\tau)$ is a global unique $p$-solution if $(v,\tau)$ is a unique maximal $p$-solution to \eqref{eq:hyper_NS_cut_off} with $\tau=\infty$ a.s. In the latter case, we simply write $v$ instead of $(v,\tau)$.
Finally, for $\g>1$, $N\geq 1$ and $p\in (2,\infty)$, we set
\begin{equation*}
\mathcal{B}_p(N)
\stackrel{{\rm def}}{=} \big\{v\in \Bs^{\g(1-2/p)}_{2,p}(\T^3)\,:\, \|v\|_{B^{\g(1-2/p)}_{2,p}(\T^3;\R^3)}\leq N\big\}
\end{equation*}
where, for convenience, we did not display the dependence on $\g>1$. Note that $\|v\|_{L^2(\T^3;\R^3)}\leq N$ if $v\in 
\mathcal{B}_p(N)$.

\begin{theorem}[Scaling limit for stochastic HNSEs with cutoff]
\label{t:scaling_limit_HNSE}
Let $1<\g<\frac{5}{4}$, $\mu>0$, $R>0$ and $N\geq 1$. Assume that $p\in (2,\infty)$ and $r>0$ satisfy
\begin{equation*}
p>\frac{4\g}{6\g-5} \qquad \text{ and }\qquad
\frac{5}{2}-2\g< r< \g\Big(1-\frac{2}{p}\Big).
\end{equation*}
Let $(\theta^n)_{n\geq 1}$ be as in \eqref{eq:property_sequence_thetan}. 
Then, for all $T<\infty$ and $r_0 <\g(1-\frac{2}{p})$, 
we have
\begin{equation}
\label{eq:scaling_limit_HNSE}
\lim_{n\to \infty}\sup_{v_0\in \mathcal{B}_p(N)} \P\Big(\sup_{t\in[0,T]}\|v^n(t)-\vd(t) \|_{H^{r_0}(\T^3;\R^3)}>\varepsilon\Big)=0  \ \text{ for all }\varepsilon>0,
\end{equation}
where $v^n$ and $\vd$ denote the unique global $p$-solution to \eqref{eq:hyper_NS_cut_off} with $\theta=\theta^n$ and to 
\begin{equation}
\label{eq:hyper_NS_cut_off_det}
\left\{
\begin{aligned}
\partial_t \vd&+\phi_{R,r}(\vd)\, \p [\nabla \cdot (\vd\otimes \vd) ]
=\frac{3\mu}{5}\Delta \vd- (-\Delta)^{\g} \vd & \text{on }&\T^3,\\
\vd(0,\cdot)&=v_0 & \text{on }&\T^3;
\end{aligned}
\right.
\end{equation}
respectively.
\end{theorem}

As it will be clear from the proof of Theorem \ref{t:main_theorem_L2} below, the possibility of choosing $r_0\geq r$ is of fundamental importance. The subcritical condition 
$p>\frac{4\g}{6\g-5}$ is used Theorem \ref{t:HNSEs_cutoff_global}.
The existence of a unique global $p$-solution to \eqref{eq:hyper_NS_cut_off} and \eqref{eq:hyper_NS_cut_off_det} is also part of the proof of Theorem \ref{t:scaling_limit_HNSE}.

Comments on the physical interpretation of the scaling limit of Theorem \ref{t:scaling_limit_HNSE} can be found \cite[Subsection 2.2]{A22}. Let us mention that the above can be regarded as a \emph{homogenization} result for the stochastic HNSEs with cutoff \eqref{eq:hyper_NS_cut_off} and \eqref{eq:hyper_NS_cut_off_det} can be thought of as the `effective' equation for the SPDE \eqref{eq:hyper_NS_cut_off} at `large scales'.

\smallskip

This section is organized as follows: In Subsection \ref{ss:global_smooth_proof}, we show how Theorem \ref{t:main_theorem_L2} can be derived from Theorem \ref{t:scaling_limit_HNSE} and the results of Sections \ref{s:local_well_stoch_problem} and \ref{s:global_high_viscosity}. In Subsections \ref{ss:global_with_cutoff} and \ref{ss:uniqueness_weak_solutions}, we analyze the global well-posedness of \eqref{eq:hyper_NS_cut_off} and \eqref{eq:hyper_NS_cut_off_det}, and discuss the uniqueness of weak solutions to \eqref{eq:hyper_NS_cut_off_det}, respectively. Finally, in Subsection \ref{ss:proof_scaling_limit_HNSE}, we prove Theorem \ref{t:scaling_limit_HNSE}.

\subsection{Global smooth solutions by transport noise -- Proof of Theorem \ref{t:main_theorem_L2}}
\label{ss:global_smooth_proof}

\begin{proof}[Proof of Theorem \ref{t:main_theorem_L2}]
By Theorem \ref{t:local_HNSEs_hyper_small_viscosity}, it is enough to show that \eqref{eq:global_high_probability} holds.
Following the proof of \cite[Theorem 1.6]{FL19}, the idea is first to prove the claim of Theorem \ref{t:main_theorem_L2} on a finite interval and second to combine the exponential decay to the mean value of solutions of the deterministic PDE \eqref{eq:hyper_NS_det} proved in Theorem \ref{t:global_high_viscosity}, and afterwards Theorem \ref{t:small_implies_global} to obtain global existence of smooth solutions. Let us mention that the second step is simpler than the one in \cite{FL19} since our scaling limit \eqref{eq:scaling_limit_HNSE} takes place in spaces of positive smoothness due to the $L^p$ in time setting with $p> 2$. 

We begin by collecting some facts. 
Set $p_\g\stackrel{{\rm def}}{=}\frac{4\g}{6\g-5}$, and let $p_\g<p<\infty$, $\varepsilon\in (0,1)$ and $N\geq 1$ be as in the statement of Theorem \ref{t:main_theorem_L2}. Moreover, let $r=r(p,\g)$ be defined via the relation $r=\frac{1}{2}(\frac{5}{2}-2\g)+ \frac{\g}{2}(1-\frac{2}{p})$. Since $p>p_\g$, it follows that  
\begin{equation}
\label{eq:assumption_exponent_main_proof}
\frac{5}{2}-2\g<r<\g\Big(1-\frac{2}{p}\Big).
\end{equation}
Next, by Theorem \ref{t:global_high_viscosity} and \eqref{eq:comparison_p_solution_etc}, there exist $\mu_0=\mu_0(p,\g,N)>0$, $\eta_0=\eta_0(p,\g)>0$ and $R_0=R_0(p,\g,N)>0$ such that, for all $u_0\in \mathcal{B}_p (N)$, there exists a unique global $p$-solution $\ud$ to \eqref{eq:hyper_NS_det} with $\mu$ replaced by $\frac{3\mu_0}{5}$ and initial data $u_0$ satisfying 
\begin{align}
\label{eq:proof_main_theorem_starting_det_1}
\sup_{t\geq 0}\|\ud(t)\|_{H^r}\leq R_0 \ \  \text{ and }\ \ 
\sup_{t\geq 0 }\Big(e^{\eta_0 t}\, \Big\|\ud(t)-\int_{\T^3}u_0(x)\,\dd x\Big\|_{H^{r}}\Big)\leq R_0.
\end{align}
Finally, let $\delta_0=\delta_0(p,\g,N,\varepsilon)>0$ be as in Theorem \ref{t:small_implies_global} with $\varepsilon_0=\varepsilon/2$, $N_0=(\sqrt{3\mu_0/2})\vee N$ and denote by $K_{p,\g}\geq 0$ the constant in the embedding $H^{r}(\T^3;\R^3)\embed B^{5/2-2\g}_{2,p_\g}(\T^3;\R^3)$. 
Without loss of generality, we may assume that $\delta_0< 2K_{p,\g}$. 

For exposition convenience, we now split the proof into two steps.
\smallskip

\emph{Step 1: For all $T>0$, there exists $\theta_0=\theta_0(p,\g,N,\varepsilon,T)\in \ell^2$ such that, for all $u_0\in \mathcal{B}_p(N)$,}
$$
\P\Big(\tau>T,\,\sup_{t\in [0,T]}\|u(t)-\ud (t)\|_{H^{r}}\leq \frac{\delta_0}{2K_{p,\g}}\Big)>1-\frac{\varepsilon}{2},
$$
\emph{where $(u,\tau)$ is the unique maximal $p$-solution to \eqref{eq:hyper_NS} with initial data $u_0$ and $\mu=\mu_0$.}
The claim of Step 1 essentially follows from the scaling limit of Theorem \ref{t:scaling_limit_HNSE}. 
Indeed, applying Theorem \ref{t:scaling_limit_HNSE} with $R=R_0+1$ as in \eqref{eq:proof_main_theorem_starting_det_1} and $\mu=\mu_0$, we obtain the existence of $\theta_0=\theta_0(p,\g,N,\varepsilon,T)\in \ell^2$ such that 
\begin{equation}
\label{eq:scaling_limit_HNSE_step1}
\sup_{u_0\in \mathcal{B}_p(N)} \P\Big(\sup_{t\in[0,T]}\|v(t)-\vd(t) \|_{H^{r}}\leq\frac{\delta_0}{2K_{p,\g}}\Big)>1-\frac{\varepsilon}{2},
\end{equation}
where $v$ and $\vd$ denote the unique global $p$-solution to \eqref{eq:hyper_NS_cut_off} with $\theta=\theta_0$ and to \eqref{eq:hyper_NS_cut_off_det} with $\mu=\mu_0$, respectively, both with initial data $v_0=u_0$. 
Now, by \eqref{eq:proof_main_theorem_starting_det_1}, 
$$
\phi_{R,r}(\ud(t))=1 \ \text{ for all }t>0.
$$
To prove that $\ud$ is the unique global $p$-solution to \eqref{eq:hyper_NS_cut_off_det}, it remains to discuss its uniqueness. Let $\vd$ be another $p$-solution to \eqref{eq:hyper_NS_cut_off_det}, and set 
$$
e_{{\rm det}}\stackrel{{\rm def}}{=}\inf\{t\in [0,+\infty)\,:\, \|\vd(t)\|_{H^{r}}\geq R\} \quad \text{ and }\quad \inf\emptyset\stackrel{{\rm def}}{=}+\infty,
$$
where we used that $B^{\g(1-2/p)}_{2,p}\embed H^{r}$ by \eqref{eq:assumption_exponent_main_proof}.
Hence, $\vd|_{[0,e_{{\rm det}})}$ is a $p$-solution of the HNSEs without cutoff \eqref{eq:hyper_NS_det} on $[0,e_{{\rm det}})$ with $\mu$ replaced by $\frac{5\mu}{3}$. By uniqueness of $\ud$, we obtain $\vd=\ud$ on $[0,e_{{\rm det}})$. Since $\sup_{t\geq 0}\|\ud(t)\|_{H^r}\leq R_0=R-1$, it readily follows that $e_{{\rm det}}=\infty$ and therefore  
\begin{equation}
\label{eq:ud_vd_equality}
\vd=\ud \ \  \text{ a.e.\ on }\ \R_+.
\end{equation}
A similar argument applies to $(u,\tau)$ and $v$. Indeed, by $\delta_0<2K_{p,\g}$, \eqref{eq:proof_main_theorem_starting_det_1} and \eqref{eq:ud_vd_equality}, 
\begin{equation}
\label{eq:v_vd_stopped_R}
\mathcal{O}_r\stackrel{{\rm def}}{=}
\Big\{\sup_{t\in[0,T]}\|v(t)-\vd(t) \|_{H^{r}}\leq \frac{\delta_0}{2K_{p,\g}}\Big\}
\subseteq
\Big\{\sup_{t\in[0,T]}\|v(t) \|_{H^{r}}< R\Big\}.
\end{equation}
The maximality of $(u,\tau)$ implies $\tau\geq e$ a.s.\ and
\begin{equation}
\label{eq:ud_vd_equality_stoc}
v=u \ \ \text{ a.e.\ on }\ [0,e)\times \O
\end{equation}
where 
$
e=\inf\{t\in [0,+\infty)\,:\, \|v(t)\|_{H^r}\geq R\} $ with $ \inf  \emptyset\stackrel{{\rm def}}{=}+\infty
$. Note that $e> T$ a.s.\ on $\mathcal{O}_r$ by \eqref{eq:v_vd_stopped_R}. Thus, the equalities 
\eqref{eq:ud_vd_equality} and \eqref{eq:ud_vd_equality_stoc} yield
$$
\P\Big(\tau>T,\,\sup_{t\in [0,T]}\|u(t)-\ud (t)\|_{H^{r}}\leq \frac{\delta_0}{2K_{p,\g}}\Big)\geq \P(\mathcal{O}_r) .
$$
The claim of Step 1 follows from \eqref{eq:scaling_limit_HNSE_step1}.

\smallskip

\emph{Step 2: Proof of \eqref{eq:global_high_probability}}. Recall that \eqref{eq:proof_main_theorem_starting_det_1} holds and let $\mu=\mu_0$. 
Choose $T_0=T_0(p,\g,N,\varepsilon)<\infty$ such that 
$
R_0 
e^{-\eta_0T_0}\leq \delta_0/(2K_{p,\g}).
$
Hence, for all $u_0\in \mathcal{B}_p(N)$, the unique global $p$-solution to the HNSEs \eqref{eq:hyper_NS_det} with increased viscosity $\mu=\frac{3\mu_0}{5}$ satisfies
$$
\Big\|\ud(T_0)-\int_{\T^3}u_0(x)\,\dd x\Big\|_{H^{r}}\leq \frac{\delta_0}{2K_{p,\g}}.
$$ 
Thus, by Step 1 applied with $T=T_0$, there exists $\theta_0=\theta_0(p,\g,N,\varepsilon)\in \ell^2$ for which the unique local $p$-solution $(u,\tau)$ to the stochastic HNSEs \eqref{eq:hyper_NS} satisfies
\begin{equation}
\label{eq:claim_T0_large_proof_main_result}
\P(\O_0)>1-\varepsilon/2\ \ \text{ with }\ \  
\O_0\stackrel{{\rm def}}{=}\big\{\tau>T_0,\, \|u(T_0)-\textstyle{\int_{\T^3}}u_0(x)\,\dd x\|_{B^{5/2-2\g}_{2,p_{\g}}}\leq \delta_0\big\},
\end{equation}
where we used $K_{p,\g}$ is the constant in the embedding $H^r\embed B^{5/2-2\g}_{2,p_\g}$ by definition.

Now, set $u_{T_0}\stackrel{{\rm def}}{=}\one_{\O_0} u(T_0) + \one_{\O\setminus \O_0} \int_{\T^3}u_0(x)\,\dd x$ and note that $\int_{\T^3} u_{T_0}(x)\,\dd x= \int_{\T^3} u_0(x)\,\dd x $ a.s. Hence, Theorem \ref{t:small_implies_global} applied with $s=T_0$ and initial data $u_{T_0}$ ensures that the stochastic HNSEs  \eqref{eq:hyper_NS_s} with $s=T_0$ and $u_s=u_{T_0}$ has a unique local $p$-solution $(\wh{u},\wh{\tau})$ such that 
$$
\P(\wh{\tau}=\infty)>1-\varepsilon/2.
$$ 
By maximality of the $p$-solution $(u,\tau)$ it follows that $\tau\geq \wh{\tau}$ a.s.\ on $\O_0$, cf.\ Definition \ref{def:p_solution}. Hence, combining the above with 
\eqref{eq:claim_T0_large_proof_main_result},
$$
\P(\tau=\infty)\geq \P\big(\O_0\cap \{\wh{\tau}=\infty\}\big)
\geq  \P(\wh{\tau}=\infty)- \P(\O\setminus\O_0) >1-\varepsilon.
$$
This proves Theorem \ref{t:main_theorem_L2} with the choice $\mu=\mu_0$ and $\theta=\theta_0$.
\end{proof}

\subsection{Global well-posedness and estimates for stochastic HNSEs with cutoff}
\label{ss:global_with_cutoff}
The aim of this subsection is to prove the following result on global well-posedness and uniform estimates w.r.t.\ normalized radially symmetric elements of $\theta\in\ell^2$ of solutions to \eqref{eq:hyper_NS_cut_off}. 

\begin{theorem}[Global well-posedness and $\theta$-uniform estimates -- Stochastic HNSEs with cutoff]
\label{t:HNSEs_cutoff_global}
Assume that $1<\g<\frac{5}{4}$ and $R>0$. Let $\theta\in \ell^2$ be normalized radially symmetric.
Suppose that $p\in (2,\infty)$ and
$r$ satisfy
\begin{equation}
\label{eq:sub_criticality_p_r}
p>\frac{4\g}{6\g-5}\qquad \text{ and }\qquad
\frac{5}{2}-2\g <r<\g\Big(1-\frac{2}{p}\Big).
\end{equation}
Then, for all $v_0\in \Bs^{\g(1-2/p)}_{2,p}(\T^3)$, the following assertions hold.
\begin{enumerate}[{\rm(1)}]
\item\label{it:HNSEs_cutoff_global_1} There exists a unique global $p$-solution to \eqref{eq:hyper_NS_cut_off} such that, a.s.\ for all $\vartheta\in [0,\frac{1}{2})$,
$$
v\in H^{\vartheta,p}_{\loc}([0,\infty);\Hs^{\g(1-2\vartheta)})\cap C([0,\infty);\Bs^{\g(1-2/p)}_{2,p}).
$$
\item\label{it:HNSEs_cutoff_global_2} For all $0<t<\infty$, there exists $C_{t,R}>0$ independent of $(v_0,\theta)$ such that
$$
\E\sup_{s\in [0,t]}\|v(s)\|_{B^{\g(1-2/p)}_{2,p}}^p
+\E\int_0^t \|v(s)\|_{H^\g}^p \,\dd s\leq C_{t,R}\big(1+\|v_0\|_{B^{\g(1-2/p)}_{2,p}}^p\big). 
$$
\end{enumerate}
\end{theorem}

The proof of the above result is given at the end of this subsection.
Before going into the details, let us comment on the conditions in \eqref{eq:sub_criticality_p_r}.
As commented on below the statement of Theorem \ref{t:local_HNSEs_hyper_small_viscosity}, the first condition in \eqref{eq:sub_criticality_p_r} is equivalent to the subcriticality of $B^{\g(1-2/p)}_{2,p}$ for the HNSEs \eqref{eq:hyper_NS} and it is necessary for having a non-trivial choice of $r$. Moreover, as for the second condition in \eqref{eq:sub_criticality_p_r}, the restriction $r>\frac{5}{2}-2\g$  ensures that the space $H^r$ is smaller than the critical space $B^{5/2-2\g}_{2,p}$ for \eqref{eq:hyper_NS}, while $r<\g(1-\frac{2}{p})$ yields that $H^r$ is larger than the space of initial data for the $p$-solutions, i.e., $B^{\g(1-2/p)}_{2,p}$.
The latter fact was used in the proof of Theorem \ref{t:main_theorem_L2} in \eqref{eq:assumption_exponent_main_proof}.
In summary, the conditions in \eqref{eq:sub_criticality_p_r} give us the following embeddings:
\begin{equation}
\label{eq:embedding_Bg_cutoff}
B^{\g(1-2/p)}_{2,p}
\embed
H^r
\embed
B^{5/2-2\g}_{2,p}.
\end{equation}
The subcritical condition $r>\frac{5}{2}-2\g$ is exploited via the following result.

\begin{lemma}[Interpolation and subcriticality]
\label{l:interpolation_subcriticality_r}
Let $1<\g<\frac{5}{4}$ and $\frac{5}{2}-2\g<r<\g$. Then, there exist $\kappa=\kappa(r,\g)\in (0,1)$ and $\coeffone=\coeffone(r,\g)\in (r,\g)$ such that, for all $v\in H^{\g}(\T^3;\R^3)$,
$$
\|\nabla \cdot (v\otimes v )\|_{H^{-\g}}\lesssim 
\|v\|_{H^{r}}^{1+\kappa} \|v\|_{H^{\coeffone}}^{1-\kappa}.
$$
\end{lemma}

The subcriticality of $H^r$ for the HNSEs \eqref{eq:hyper_NS} is encoded in the sublinear growth of the RHS of the estimate in the above result with respect to the $H^{\g}$-norm.

\begin{proof}
Recall from Lemma \ref{lem:nonlinearity_estimate} that, for $\beta=\frac{5}{4}-\frac{\g}{2}$,
\begin{align*}
\|\nabla \cdot (v\otimes v )\|_{H^{-\g}}
\lesssim_\g \|v\|_{H^{\beta}}^2\ \  \text{ for all }\ v\in H^\beta.
\end{align*}
It is clear that $\beta>\frac{5}{2}-2\g$ and $\beta<\g$. Moreover, if $r\geq \beta$, then the claim of Lemma \ref{l:interpolation_subcriticality_r} holds for any $\kappa\in (0,1)$ and all $r< \coeffone<\g$. 
It remains to discuss the case $r<\beta$. 
Without loss of generality, we may assume $\coeffone>\beta$.
Thus, we can write $[H^r,H^{\coeffone}]_{\eta}=H^{\beta}$ with $\eta=(\beta-r)/(\coeffone-r)$ and therefore
$$
\|v\|_{H^{\beta}}^2\lesssim \|v\|_{H^r}^{2(1-\eta)}\|v\|_{H^{\coeffone}}^{2\eta}
$$
for all $v\in H^\g(\T^3;\R^3)$. Note that
$$
2\eta= 1-\kappa \quad \text{ with }\quad \kappa(r,\coeffone)= \frac{\coeffone+r-\frac{5}{2}+\g}{\coeffone-r}.
$$
In particular, $\lim_{\coeffone\uparrow \g} \kappa(r,\coeffone)= (r+2\g-\frac{5}{2})/(\g-r)$. 
The claim follows by noticing that $ (r+2\g-\frac{5}{2})/(\g-r)\in (0,1)$. Indeed, the positivity follows from the assumption $r<\beta$, and the remaining condition from subcriticality, i.e., $r>\frac{5}{2}-2\g$.
\end{proof}

\begin{proof}[Proof of Theorem \ref{t:HNSEs_cutoff_global}]
We split the proof into three steps.

\smallskip

\emph{Step 1: There exists a unique maximal $p$-solution $(v,\eta)$ to \eqref{eq:hyper_NS_cut_off}. Moreover, the following blow-up criterion holds:}
\begin{equation}
\label{eq:blow_up_criterion_v}
\P\Big(\eta<t,\, \sup_{s\in [0,\eta)}\|v(s)\|_{B^{\g(1-2/p)}_{2,p}}^p
+\int_0^\eta \|v(s)\|_{H^\g}^p\,\dd s <\infty\Big)=0
\ \ 
\text{\emph{for all} $t<\infty$}.
\end{equation}
In this step, similar to the proof of Theorem \ref{t:local_HNSEs_hyper_small_viscosity} in Subsection \ref{ss:proofs_local}, we employ the abstract results in \cite{AV19_QSEE1,AV19_QSEE2} (see also \cite[Step 1 of Lemma 4.3]{A24_anomalous} for a similar situation).
As in Theorem \ref{t:local_HNSEs_hyper_small_viscosity}, we regard \eqref{eq:hyper_NS_cut_off} as a stochatic evolution equation of the form \eqref{eq:SEE} on $X_0=\Hs^{-\g}$, with $(A,B)$ as in \eqref{eq:choice_ABF} and 
$$
F(v)=-\phi_{R,r}(v)\,\p[\nabla \cdot(v\otimes v)].
$$
Let $X_1=\Hs^\g$. 
We claim that, for all $v,v'\in X_1$,
\begin{align}
\label{eq:estimate_truncated_F_v_revision}
\|F(v)-F(v')\|_{H^{-\g}}&\lesssim_R (1+\|v\|_{H^{\beta}}+\|v'\|_{H^{\beta}})\|v-v'\|_{H^{\beta}}\\
\nonumber
&
\,+ (1+\|v\|_{H^{\coeffone}}+\|v'\|_{H^{\coeffone}})\|v-v'\|_{H^r}.
\end{align}
Let us first discuss how the above and the results in \cite{AV19_QSEE1,AV19_QSEE2} imply the claim of Step 1. Due to the first embedding in \eqref{eq:embedding_Bg_cutoff} and $\coeffone<\g$, the conditions of \cite[Theorem 4.8]{AV19_QSEE1} are satisfied. Hence, the existence of a unique local $p$-solution to \eqref{eq:hyper_NS_cut_off} follows from \cite[Theorem 4.8]{AV19_QSEE1}. Meanwhile, the blow-up criterion in \eqref{eq:blow_up_criterion_v} follows from the above considerations and \cite[Theorem 4.10(3)]{AV19_QSEE2} applied with trivial time weight.

In the remaining part of this step, we prove \eqref{eq:estimate_truncated_F_v_revision}. Below, $v,v'\in \Hs^\g$ are fixed. Without loss of generality, we assume $\|v\|_{H^r}\geq \|v'\|_{H^r}$. Note that 
$$
F(v)-F(v')= \phi_{R,r}(v)(\Bi(v)-\Bi(v'))+ (\phi_{R,r}(v)-\phi_{R,r}(v'))\, \Bi(v'),
$$
where $\Bi(v)\stackrel{{\rm def}}{=}-\,\p[\nabla\cdot (v\otimes v)]$.
Lemma \ref{lem:nonlinearity_estimate} and $\phi_{R,r}(v)\leq 1$ yield
\begin{align*}
\| F(v)-F(v')\|_{X_0}
&\leq \|\Bi(v)-\Bi(v')\|_{X_0}+ \| (\phi_{R,r}(v)-\phi_{R,r}(v'))\,\Bi(v')\|_{X_0}\\
&\lesssim(1+\|v\|_{H^{\beta}}+\|v'\|_{H^{\beta}})\|v-v'\|_{H^{\beta}}+ |\phi_{R,r}(v)-\phi_{R,r}(v')|\, \|\Bi(v')\|_{X_0}.
\end{align*}
Hence, to obtain \eqref{eq:estimate_truncated_F_v_revision}, it remains to estimate the last term on the right-hand side of the above inequality. As $\|v\|_{H^r}\geq \|v'\|_{H^r}$, if $\|v'\|_{H^r}\geq 2R$, then 
\begin{equation*}
\big|\phi_{R,r}(v)-\phi_{R,r}(v')\big|=0.
\end{equation*}
Let $\kappa>0$ and $\coeffone<\g$ be as in Lemma \ref{l:interpolation_subcriticality_r}. From the latter and above formula, we obtain
\begin{align*}
|\phi_{R,r}(v)-\phi_{R,r}(v')|\, \|\Bi(v')\|_{X_0}
& \lesssim 
|\phi_{R,r}(v)-\phi_{R,r}(v')|\, \|v'\|_{H^r}^{1+\kappa}\|v'\|_{H^{\coeffone}}^{1-\kappa}\\
&\lesssim_R
|\phi_{R,r}(v)-\phi_{R,r}(v')| \,\|v'\|_{H^{\coeffone}}^{1-\kappa}\\
&
\lesssim_R \|v-v'\|_{H^r}
(1+ \|v'\|_{H^{\coeffone}}),
\end{align*}
where in the last step we used \eqref{eq:definition_cutoff_phiRr}, and that $\phi$ is smooth with compact support.
The claimed estimate \eqref{eq:estimate_truncated_F_v_revision} now follows from collecting the above inequalities.

\smallskip

\emph{Step 2: For all $t>0$, there exists $C_{t,R}>0$ independent of $(v_0,\theta)$ such that}
$$
\E\Big[\sup_{s\in [0,t\wedge \eta)}\|v(s)\|_{B^{\g(1-2/p)}_{2,p}}^p\Big]
+\E\int_0^{t\wedge \eta}\|v(s)\|_{H^\g}^p \,\dd s\leq C_{t,R}\|v_0\|_{B^{\g(1-2/p)}_{2,p}}^p. 
$$
Let us fix $t<\infty$.
For all $j\geq 1$, let $\eta_j$ be the stopping time defined as
$$
\eta_j \stackrel{{\rm def}}{=}
\inf\Big\{r\in [0,\eta) \,:\,  \sup_{s\in [0,r)}\|v(s)\|_{B^{\g(1-2/p)}_{2,p}}^p
+\int_0^r \|v(s)\|_{H^\g}^p\,\dd s\geq j\Big\}\wedge t
$$
where we set $\inf\emptyset \stackrel{{\rm def}}{=}\infty$. 

By Theorem \ref{t:smr_theta}\eqref{it:smr_theta_1}, there exists a constant $C_0(t)>0$ independent of $(v_0,\theta)$ such that, for all $j\geq 1$,
\begin{align*}
&\E\Big[\sup_{s\in [0,\eta_j)}\|v(s)\|_{B^{\g(1-2/p)}_{2,p}}^p\Big]
+\E\int_0^{ \eta_j}\|v(s)\|_{H^\g}^p \,\dd s\\
&\qquad\leq C_0(t)\Big(\|v_0\|_{B^{\g(1-2/p)}_{2,p}}^p
+ \E\int_0^{ \eta_j}\phi_{R,r}(v)\,\|\p[\nabla \cdot (v(s)\otimes v(s))]\|_{H^{-\g}}^p \,\dd s\Big)\\
&\qquad\stackrel{(i)}{\leq} C_0(t)\Big(\|v_0\|_{B^{\g(1-2/p)}_{2,p}}^p
+ (2R)^{p(1+\kappa)}\E\int_0^{ \eta_j}\|v(s)\|_{H^{\g}}^{p(1-\kappa)} \,\dd s\Big)\\
&\qquad\stackrel{(ii)}{\leq} C_0(t)\big(\|v_0\|_{B^{\g(1-2/p)}_{2,p}}^p
+ C_{R,\kappa}\big)+\frac{1}{2} \,\E\int_0^{ \eta_j}\|v(s)\|_{H^{\g}}^{p} \,\dd s
\end{align*}
where in $(i)$ we used Lemma \ref{l:interpolation_subcriticality_r} as well as $\phi_{R,r}(v)=0$ provided $\|v\|_{H^r}\geq 2R$, and in $(ii)$ the Young inequality. Since $\kappa=\kappa(r,\g)>0$ by Lemma \ref{l:interpolation_subcriticality_r}, in the above $C_{R,\kappa}=C_{R,\g,r}$ is a costant depending only on $R,r$ and $\g$.

Noticing that $\E\int_0^{ \eta_j}\|v(s)\|_{H^\g}^p \,\dd s\leq j<\infty$ by definition of $\eta_j$, the above implies
$$
\E\Big[\sup_{s\in [0,\eta_j)}\|v(s)\|_{B^{\g(1-2/p)}_{2,p}}^p\Big]
+\E\int_0^{ \eta_j}\|v(s)\|_{H^\g}^p \,\dd s\leq C_{t,R}\big(1+\|v_0\|_{B^{\g(1-2/p)}_{2,p}}^p\big)
$$
where $C_{t,R}=2C_{R,r,\g}C_0(t)$ is independent of $(v_0,\theta)$. Since $\lim_{j\to \infty}\eta_j =\eta\wedge t$ a.s.\ by construction, the claim of Step 2 follows by taking the limit as $j\to \infty$ and using the Fatou lemma in the above estimate.

\smallskip

\emph{Step 3: Conclusion -- $\eta=\infty$ a.s.} By Step 2, we have that, for all $t<\infty$, 
$$
\sup_{s\in [0,t\wedge \eta)}\|v(s)\|_{B^{\g(1-2/p)}_{2,p}}^p+ 
\int_0^{t\wedge \eta}\|v(s)\|_{H^\g}^p \,\dd s<\infty\  \ \text{ a.s.\ }
$$
Thus, for all $t<\infty$,
$$
\P(\eta<t)=\P\Big(\eta<t,\, \sup_{s\in [0,\eta)}\|v(s)\|^p_{B^{\g(1-2/p)}_{2,p}}
+\int_0^\eta \|v(s)\|_{H^\g}^p\,\dd s <\infty\Big)=0
$$
where in the last step we used the blow-up criterion of Step 1. Hence, the arbitrariness of $t<\infty$ yields $\eta=\infty$ a.s. Thus, the assertions \eqref{it:HNSEs_cutoff_global_1} and \eqref{it:HNSEs_cutoff_global_2} in Theorem \ref{t:HNSEs_cutoff_global} 
follow from the latter and Steps 1 and 2, respectively.
\end{proof}

\subsection{Uniqueness of weak solutions of HNSEs with viscosity and cutoff}
\label{ss:uniqueness_weak_solutions}
In this subsection, we analyze the regularity of certain weak solutions to \eqref{eq:hyper_NS_cut_off_det}
arising in the scaling limit behind Theorem \ref{t:scaling_limit_HNSE}.

\begin{definition}[Weak solution of HNSEs with viscosity and cutoff]
\label{def:HNSE_cutoff}
Fix $R>0$ and $r\in (0,\g)$. Let $v_0\in \Hs^r(\T^3)$ and 
$$
\vd \in C([0,T];\Hs^r(\T^3))\cap L^2(0,T;\Hs^\g(\T^3)) 
$$
for some $T<\infty$. We say that $\vd$ is a weak solution to \eqref{eq:hyper_NS_cut_off_det} if for all divergence-free vector field $\varphi\in C^{\infty}(\T^3;\R^3)$ and $t\in (0,T)$ we have
\begin{align*}
\langle \varphi , \vd(t)\rangle 
=
\langle \varphi , v_0\rangle 
&+\int_0^t\int_{\T^3}\Big[\,
\frac{3\mu}{5}\, \vd\Delta \varphi-  \vd(-\Delta)^{\g}\varphi\Big] \,\dd x \,\dd s\\
&+\int_0^t\phi_{R,r}(\vd)\,\int_{\T^3}
\big[  (\vd\otimes \vd) : \nabla \varphi\big]\,\dd x \,\dd s.
\end{align*}
\end{definition}

The aim of this subsection is to prove the following result.

\begin{proposition}[Global well-posedness and weak-strong uniqueness for HNSEs with viscosity and cutoff]
\label{prop:regularity_weak_sol_cut_off}
Let $1<\g<\frac{5}{4}$ and $\frac{5}{2}-2\g<r<\g$. Fix $p\in (2,\infty)$ such that 
$
r<\g(1-2/p)
$
and assume
\begin{equation}
\label{eq:v_0_Besov_high_regularity_vd}
v_0\in \Bs^{\g(1-2/p)}_{2,p}.
\end{equation}
Then there exists a unique global $p$-solution $\vd$ to \eqref{eq:hyper_NS_cut_off_det}. Moreover, for all weak solutions $w_{{\rm det}}$ to \eqref{eq:hyper_NS_cut_off_det} on some time interval $[0,T]$, we have $\vd =w_{{\rm det}}$ on $[0,T]$. 
\end{proposition}

The proof of Proposition \ref{prop:regularity_weak_sol_cut_off} is given at the end of this subsection and relies on the following result, which shows that weak solutions of \eqref{eq:hyper_NS_cut_off_det} are actually in the regularity class of $p$-solutions. 

\begin{lemma}[Regularity of weak solutions of HNSEs with transport noise and cutoff]
\label{l:regularity_weak_sol_cut_off}
Let the assumptions of Proposition \ref{prop:regularity_weak_sol_cut_off} be satisfied. 
Let $\wwd$ be a weak solution to \eqref{eq:hyper_NS_cut_off_det} on $[0,T]$ for some $T<\infty$. Then
$$
\wwd \in W^{1,p}(0,T;\Hs^{-\g})\cap L^{p}(0,T;\Hs^{\g})\subseteq C([0,T];\Bs^{\g(1-2/p)}_{2,p}).
$$
\end{lemma}

\begin{proof}
Let us begin by recalling that, from Lemma \ref{l:interpolation_subcriticality_r}, there exists $\kappa\in (0,1)$ such that, for all $v\in H^\g$,
\begin{equation}
\label{eq:v_estimate_proof_subcritical_improved_regularity}
\|\nabla \cdot (v\otimes v )\|_{H^{-\g}}\lesssim 
\|v\|_{H^{r}}^{1+\kappa} \|v\|_{H^{\g}}^{1-\kappa}.
\end{equation}
Since $\vd\in L^2(0,T;H^\g)$ by Definition \ref{def:HNSE_cutoff}, to prove the claim of Proposition \ref{prop:regularity_weak_sol_cut_off}, it is enough to show that, for $j\geq 1$,
\begin{equation}
\label{eq:implication_regularity_weak_solutions_det_1}
\vd\in L^{p_{j-1}}(0,T;H^{\g})\ 
\Longrightarrow\ 
\vd \in W^{1,p_j}(0,T;H^{-\g})\cap L^{p_j}(0,T;H^{\g})
\end{equation}
where 
\begin{equation}
\label{eq:def_pj_iteration}
\textstyle
p_j= p\wedge  \big( \frac{2}{(1-\kappa)^j}\big).
\end{equation}
To prove \eqref{eq:implication_regularity_weak_solutions_det_1}, we employ (deterministic) maximal $L^p$-regularity for the operator $-(-\Delta)^{\g}$ on $\Hs^{-\g}$ with forcing term (see \cite[Subsection 3.3.5]{pruss2016moving} for the terminology) 
$$
f=\phi_{R,r}(v)\,\p [\nabla\cdot (v\otimes v)].
$$
Note that maximal $L^p$-regularity estimates for $-(-\Delta)^{\g}$ on $\Hs^{-\g}$ corresponds to the case $\xi_n\equiv 0$ and $g_n\equiv 0$ in Theorem \ref{t:smr_theta}. 

Assume 
$
\vd\in L^{p_{j_0-1}}(0,T;H^{\g})$ for some $j_0\geq 1$.
By \eqref{eq:v_estimate_proof_subcritical_improved_regularity}, 
for all $v\in H^\g$,
\begin{align*}
\phi_{R,r}(v)\,\| \nabla \cdot (v\otimes v)\|_{H^{-\g}}
\lesssim 
\phi_{R,r}(v)\,\|v\|_{H^r}^{1+\kappa}\|v\|_{H^\g}^{1-\kappa}
\lesssim_R \|v\|_{H^\g}^{1-\kappa},
\end{align*} 
where we used that $\phi_{R,r}(v)=0$ if $\|v\|_{H^r}\geq 2R$.
Note also that $u_0\in \Bs^{\g(1-2/p_{j_0-1})}_{2,p_{j_0-1}}$ due to \eqref{eq:v_0_Besov_high_regularity_vd} and \eqref{eq:def_pj_iteration}. 
Hence, from the above and the inductive assumption,  
$$
\phi_{R,r}(v)\,\p [\nabla \cdot (v\otimes v)]
\in L^{p_{j_0}}(0,T;H^{-\g}).
$$
Using the above and \eqref{eq:v_0_Besov_high_regularity_vd} in combination with the maximal $L^p$-regularity for $-(-\Delta)^{\g}$ on $\Hs^{-\g}$ as in Theorem \ref{t:smr_theta} with $\xi_n\equiv 0$ and $g_n\equiv 0$, we obtain $\vd \in W^{1,p_{j_0}}(0,T;H^{-\g})\cap L^{p_{j_0}}(0,T;H^{\g})$ as desired.
\end{proof}

\begin{proof}[Proof of Proposition \ref{prop:regularity_weak_sol_cut_off}]
The existence of a unique global $p$-solution $\vd$ to \eqref{eq:hyper_NS_cut_off_det} follows as in the proof of Theorem \ref{t:HNSEs_cutoff_global}\eqref{it:HNSEs_cutoff_global_1} and therefore we omit the details. Now, by Lemma \ref{l:regularity_weak_sol_cut_off}, weak solutions to \eqref{eq:hyper_NS_cut_off_det} are actually local $p$-solutions to \eqref{eq:hyper_NS_cut_off_det}. 
Thus, the claim follows from the uniqueness of the $p$-solution $\vd$.
\end{proof}

Before going further, let us point out that Lemma \ref{l:regularity_weak_sol_cut_off} still holds if we relax the regularity assumptions in Definition \ref{def:HNSE_cutoff} to $v_0\in \Ls^q(\T^3)$ and $\vd\in L^\infty(0,T;\Ls^{q}(\T^3))$ for some $q>\frac{3}{2\g-1}$. The latter value is the critical integrability for the HNSEs and can be obtained via a scaling argument as in \eqref{eq:scaling_HNSEs}-\eqref{eq:u0_u0lambda} of Subsection \ref{ss:novelty}. The possibility of reducing the regularity assumptions to the above-mentioned ones is consistent with Serrin's type criteria (see e.g., \cite[Proposition 12.2 and Theorem 12.4]{LePi}) and the subcriticality of $\Ls^{q}(\T^3)$ for the HNSEs. As in Remark \ref{r:generalization}, for the sake of clarity, we do not pursue this level of generality here. 

\subsection{Scaling limit -- Proof of Theorem \ref{t:scaling_limit_HNSE}}
\label{ss:proof_scaling_limit_HNSE}
The key ingredient in the proof of Theorem \ref{t:scaling_limit_HNSE} is the following result.

\begin{proposition}
\label{prop:scaling_limit_HNSE_reduction}
Let $1<\g<\frac{5}{4}$, $\frac{4\g}{6\g-5}<p<\infty$, $R>0$ and $\frac{5}{2}-2\g< r < \g(1-\frac{2}{p})$. 
Let $(\theta^n)_{n\geq 1}$ be as in \eqref{eq:choice_theta_n}. 
Assume that $
(v_0^n)_{n\geq 1}\subseteq \Bs^{\g(1-2/p)}_{2,p}$ is a sequence satisfying 
$$
v_0^n \to v_0\, \text{ weakly in $ B^{\g(1-2/p)}_{2,p} $
for some $v_0\in B^{\g(1-2/p)}_{2,p}$}.
$$
Then, for all $T<\infty$ and $r_0 <\g(1-\frac{2}{p})$,
\begin{equation}
\label{eq:scaling_limit_HNSE_reduction}
\lim_{n\to \infty} \P\Big(\sup_{t\in [0,T]}\|v^n(t)-\vd^n (t)\|_{H^{r_0}}>\varepsilon\Big)=0 \ \ \text{ for all }\varepsilon>0,
\end{equation}
where $v^n$ and $\vd^n$ denote the unique global $p$-solution to \eqref{eq:hyper_NS_cut_off} with $\theta=\theta^n$  and to \eqref{eq:hyper_NS_cut_off_det} with initial data $v_0^n$, respectively.
\end{proposition}

Let us first show how the above result implies Theorem \ref{t:scaling_limit_HNSE}. The proof of Proposition \ref{prop:scaling_limit_HNSE_reduction} is given at the end of this subsection.

\begin{proof}[Proof of Theorem \ref{t:scaling_limit_HNSE}]
We prove Theorem \ref{t:scaling_limit_HNSE} by contradiction. Suppose that there exist $\varepsilon>0$, $T>0$ and $r_0<\g(1-\frac{2}{p})$ for which \eqref{eq:scaling_limit_HNSE} does not hold. 
Thus, there exist $\varepsilon>0$ and a (not relabeled sub)sequence $(v_0^n)_{n\geq 1}\subseteq \Bs^{\g(1-2/p)}_{2,p}$ such that $\|v_0^n\|_{B^{\g(1-2/p)}_{2,p}}\leq N$ and
\begin{equation}
\label{eq:contradiction_limit_proof_scaling}
\lim_{n\to \infty}\P\Big(\sup_{t\in [0,T]}\|v^n(t)-\vd^n(t) \|_{H^{r_0}(\T^3;\R^3)}>\varepsilon\Big)>0
\end{equation}
where $v^n$ and $\vd^n$ denote the unique global $p$-solution to \eqref{eq:hyper_NS_cut_off} with $\theta=\theta^n$  and to \eqref{eq:hyper_NS_cut_off_det} with initial data $v_0^n$, respectively.
Up to extract a subsequence, we may assume that $v^n_0\to v_0$ weakly in $B^{\g(1-2/p)}_{2,p}$.
Hence, \eqref{eq:contradiction_limit_proof_scaling} contradicts the claim of Proposition \ref{prop:scaling_limit_HNSE_reduction}. Therefore, the claim of Theorem \ref{t:scaling_limit_HNSE} is proved.
\end{proof}

For the proof of Proposition \ref{prop:scaling_limit_HNSE_reduction}, we need the following preliminary result.

\begin{lemma}[$\theta$-uniform time-regularity estimate]
\label{l:time_regularity_estimate}

Let $1<\g<\frac{5}{4}$, $\frac{4\g}{6\g-5}<p<\infty$, $R>0$ and $\frac{5}{2}-2\g< r < \g(1-\frac{2}{p})$.  Fix $a\in (1,\infty)$ and $u_0\in \Bs^{\g(1-2/p)}_{2,p}(\T^3)$. Let $v$ be the unique global $p$-solution to \eqref{eq:hyper_NS_cut_off} provided by Theorem \ref{t:HNSEs_cutoff_global}. Let $T<\infty$ and set  
$$
\mathcal{M}(t)\stackrel{{\rm def}}{=}\sqrt{\frac{3\mu}{2}}\sum_{k,\alpha} \theta_{k}\int_0^t\p[ (\sigma_{k,\alpha}\cdot\nabla) v] \,\dd W^{k,\alpha}_t \ \  \text{ for } \  t\in [ 0,T].
$$
Then there exist $s_1,r_1,K_1>0$ independent of $(u_0,\theta)$ such that 
\begin{align}
\label{eq:time_regularity_estimate_1}
\E\big[\|\mathcal{M}\|_{C^{s_1}(0,T;H^{-r_1})}^{2a}\big]
&\leq K_1\|\theta\|_{\ell^{\infty}}^{2a}\|u_0\|_{B^{\g(1-2/p)}_{2,p}}^{2a},\\
\label{eq:time_regularity_estimate_2}
\E\big[\|v\|_{C^{s_1}(0,T;H^{-r_1})}^{2a}\big]
&\leq K_1\|u_0\|_{B^{\g(1-2/p)}_{2,p}}^{2a}.
\end{align}
\end{lemma}

The key point in the above result is the presence of $\|\theta\|_{\ell^{\infty}}$ on the RHS\eqref{eq:time_regularity_estimate_1}.
Lemma \ref{l:time_regularity_estimate} is well-known to experts, and it is a consequence of the structure of the noise described in Subsection \ref{ss:noise} and the energy equality for $v$:
\begin{equation}
\label{eq:energy_inequality_again_revised_cutoff}
\frac{1}{2}\|v(t)\|_{L^2}^2 + \int_0^t\|(-\Delta)^\g v\|_{L^2}^2 \,\dd s =\frac{1}{2}\|v_0\|_{L^2}^2 \  \text{ a.s.\ for all }t<\infty.
\end{equation}
As for \eqref{eq:energy_balance_v_revision}, the above is a consequence of the It\^o formula \cite[Theorem 4.2.5]{LR15}  (see also the comments below \eqref{eq:energy_balance_v_revision}) and the fact that the truncation of the nonlinearity $\phi_{R,r}(v)$ in \eqref{eq:hyper_NS_cut_off} is independent of $x\in\T^3$.
 We omit the details, and we refer to the proof of either \cite[Proposition 3.6]{FGL21} or \cite[Lemma 6.3]{A22} for a similar situation.

\begin{proof}[Proof of Proposition \ref{prop:scaling_limit_HNSE_reduction} -- Sketch]
The argument of the current proof follows essentially the one in the proof of \cite[Theorem 1.4]{FL19} (see also \cite[Proposition 3.7]{FGL21} and \cite[Theorem 6.1]{A22}). Thus, we only provide a sketch. Without loss of generality, we assume $r_0\geq r$.
In the following, we show that for each (not relabeled sub)sequence of $(v^n)_{n\geq 1}$, we can find a further subsequence for which \eqref{eq:scaling_limit_HNSE_reduction} holds.
Let $0<T<\infty$ be fixed.
By the Ascoli-Arzel\'a theorem and standard interpolation arguments, it is clear that, for all $r_0<\g(1-2/p)$ and $r_1,s_1>0$,  the following embedding is compact
\begin{equation}
\label{eq:compactness_path_space}
C([0,T];B^{\g(1-2/p)}_{2,p})\cap C^{s_1}(0,T;H^{-r_1})\embed_{{\rm c}} C([0,T];H^{r_0}).
\end{equation}
Now, let $(\theta^n)_{n\geq 1}$ be as in \eqref{eq:choice_theta_n}.
By Theorem \ref{t:HNSEs_cutoff_global}\eqref{it:HNSEs_cutoff_global_2} and \eqref{eq:time_regularity_estimate_2} in Lemma \ref{l:time_regularity_estimate}, there exists $s_1,r_1>0$ such that
$$
\sup_{n\geq 1}\big(
\E \sup_{t\in [0,T]}\|v^n(t)\|_{B^{\g(1-2/p)}_{2,p}}^2 + \E\|v^n\|_{C^{s_1}(0,T;H^{-r_1})}^2\big)<\infty.
$$
In particular, by
Prokhorov’s theorem and \eqref{eq:compactness_path_space}, the sequences of laws $(\mu_n)_{n\geq1}$, where $\mu_n\stackrel{{\rm def}}{=}\mathcal{L}_{v^n}$, admits a weakly convergence (not relabeled) subsequence in the space $C([0,T];\Hs^{r_0})$.  
Let us denote by $\mu$ its weak limit. Note that, by the energy inequality \eqref{eq:energy_inequality_again_revised_cutoff} and $\sup_{n\geq 1}\|v_0^n\|_{B^{\g(1-2/p)}_{2,p}}<\infty$ by assumption, there exists a deterministic constant $K\geq 1$ such that   
$
\displaystyle{\sup_{n\geq 1}\int_0^T\|v^n(t)\|^2_{H^\g}\,\dd t\leq K} 
$
a.s.
Let 
\begin{equation}
\label{eq:def_X_K}
\mathcal{X}_K\stackrel{{\rm def}}{=}\Big\{v\in  C([0,T];\Hs^{r_0})\,:\, \int_0^T\|v\|^2_{H^\g}\,\dd t\leq K\Big\}.
\end{equation}
Note that $\mathcal{X}_K\subset C([0,T];\Hs^{r_0})$ is closed due to the lower semicontinuity of the $L^2(0,T;H^\g)$-norm. 
Since $v^n\in \mathcal{X}_K$ for all $n\geq 1$, it follows that $\mu(\mathcal{X}_K)=1$.

Next, let us recall that, by \cite[Theorem 5.1]{FL19},    
for all divergence-free $\varphi\in C^{\infty}(\T^3;\R^3)$,
\begin{equation}
\label{eq:additional_dissipation_varphi_proof}
\lim_{n\to \infty} \LLn \varphi = \frac{3\mu}{5}\Delta\varphi \ \text{ in }\ L^2(\T^3;\R^3)
\end{equation}
where $\LLn \stackrel{{\rm def}}{=} \LLnn$ is the It\^o-Stratonovich corrector defined in \eqref{eq:Ito_stratonovich_change}.

Now, arguing as in \cite[Proposition 4.2]{FL19} (see also Step 1 in \cite[Theorem 6.1]{A22} or \cite[Proposition 3.7]{FGL21}), by combining \eqref{eq:time_regularity_estimate_1}, the continuity of the map
\begin{align*}
\mathcal{X}_K&\to C([0,T];\R), \\
v&\mapsto \Big[ t\mapsto \int_0^t\phi_{R,r}(v)\int_{\T^3} (v\otimes v) :\nabla \varphi \,\dd x\,\dd s\Big],
\end{align*}
for all divergence-free vector field $\varphi\in C^{\infty}(\T^d;\R^d)$ and \eqref{eq:additional_dissipation_varphi_proof}, it follows that 
$$
\mu\Big(v\in \mathcal{X}_K\,:\,v \text{ is a weak solution to }\eqref{eq:hyper_NS_cut_off_det}\Big)=1.
$$
For the definition of weak solutions to \eqref{eq:hyper_NS_cut_off_det}, see Definition \ref{def:HNSE_cutoff}.
Now, by \eqref{eq:def_X_K} and Proposition \ref{prop:regularity_weak_sol_cut_off}, it follows that $\mu=\delta_{\vd}$ where $\vd$ is
unique global $p$-solution to \eqref{eq:hyper_NS_cut_off_det} with initial data $v_0$.
Since the convergence $v^n \to \vd$ in distribution in $C([0,T];H^{r_0})$ implies convergence in probability in case of a deterministic limit, cf.\ Step 2 in \cite[Theorem 6.1]{A22}, 
we have
\begin{equation*}
\lim_{n\to \infty} \P\Big(\sup_{t\in [0,T]}\|v^n(t)-\vd (t)\|_{H^{r_0}}>\varepsilon\Big)=0 \ \ \text{ for all }\varepsilon>0.
\end{equation*}
The claim of Proposition \ref{prop:scaling_limit_HNSE_reduction} now follows by noticing that 
$$
\vd^n\to \vd \ \text{ in }C([0,T];H^{r_0}).
$$ 
To prove the above, it is enough to repeat the argument in the proof of Theorem \ref{t:HNSEs_cutoff_global}\eqref{it:HNSEs_cutoff_global_2} to obtain a uniform in $n$ estimate in $L^p(0,T;H^\g)\cap C([0,T];B^{\g(1-2/p)}_{2,p})$ for $\vd^n$ and conclude using Proposition \ref{prop:regularity_weak_sol_cut_off}, \eqref{eq:compactness_path_space} and the energy balance.
\end{proof}

\subsubsection*{Acknowledgements}
The author thanks Federico Butori, Umberto Pappalettera, and Mark Veraar for useful comments and suggestions. 
The author is grateful to the anonymous referees for their useful comments and a careful reading of the manuscript.

\def\polhk#1{\setbox0=\hbox{#1}{\ooalign{\hidewidth
  \lower1.5ex\hbox{`}\hidewidth\crcr\unhbox0}}} \def\cprime{$'$}

\end{document}